\documentclass[11pt]{article}
\usepackage{amsthm, amsmath, amssymb, amsfonts, url, booktabs, tikz, setspace}
\usepackage[margin = 1in]{geometry}

\usepackage{hyperref}

\hypersetup{
    pdfstartview={FitH}    
}

\newtheorem*{claim*}{Claim}
\newtheorem*{fact*}{Fact}
\newtheorem{theorem}{Theorem}[section]
\newtheorem*{theorem*}{Theorem}
\newtheorem{proposition}[theorem]{Proposition}
\newtheorem{lemma}[theorem]{Lemma}
\newtheorem*{lemma*}{Lemma}
\newtheorem{corollary}[theorem]{Corollary}
\newtheorem{conjecture}[theorem]{Conjecture}

\theoremstyle{definition}
\newtheorem{definition}[theorem]{Definition}
\newtheorem{example}[theorem]{Example}

\newtheorem*{question}{Question}

\theoremstyle{remark}
\newtheorem*{remark}{Remark}

\def\b{\mathbf}

\def\c{\mathcal}

\def\RR{\mathbb{R}}

\def\ZZ{\mathbb{Z}}

\def\bx{\mathbf{x}}

\DeclareMathOperator{\Hom}{Hom}
\DeclareMathOperator{\vol}{vol}
\DeclareMathOperator{\conv}{conv}
\DeclareMathOperator{\swap}{swap}
\DeclareMathOperator{\viol}{viol}
\DeclareMathOperator{\pair}{pair}

\newcommand{\bsp}{\mathrm{bsp}}
\newcommand{\bst}{\mathrm{bst}}
\newcommand{\surj}{\mathrm{surj}}

\newcommand{\abs}[1]{\left\lvert#1\right\rvert}

\newcommand{\set}[1]{\left\{ #1 \right\}}
\def\({\left(}
\def\){\right)}

\newcommand{\STAB}{\mathsf{STAB}}

\newcommand{\ESTAB}{\mathsf{ESTAB}}

\def\longto{\longrightarrow}

\def\x{\times}

\def\<={\Leftarrow}
\def\=>{\Rightarrow}

\tikzstyle{P} = [draw, circle, black, fill, inner sep = 0pt, minimum width = 3pt]
\tikzstyle{t} = [ultra thick]
\tikzstyle{every loop} = []
\tikzstyle{sf} = [font={\small}]

\title{The Bipartite Swapping Trick on Graph Homomorphisms}
\author{Yufei Zhao \\
\small Massachusetts Institute of Technology\\
\small \url{yufei.zhao@gmail.com}}

\begin{document}

%

\maketitle

\begin{abstract}
We provide an upper bound to the number of graph homomorphisms from $G$ to $H$, where $H$ is a fixed graph with certain properties, and $G$ varies over all $N$-vertex, $d$-regular graphs. This result generalizes a recently resolved conjecture of Alon and Kahn on the number of independent sets. We build on the work of Galvin and Tetali, who studied the number of graph homomorphisms from $G$ to $H$ when $H$ is bipartite. We also apply our techniques to graph colorings and stable set polytopes.
\end{abstract}

\section{Introduction} \label{sec:intro}

\subsection{From independent sets to graph homomorphisms} \label{sec:intro-indep}

Let $G = (V,E)$ be a (simple, finite, undirected) graph. An \emph{independent set} (or a \emph{stable set}) is a subset of the vertices with no two adjacent. Let $i(G)$ denote the number of independent sets of $G$. The following question is motivated by applications in combinatorial group theory \cite{Alon91, Zhao:finite} and statistical mechanics \cite{Kahn}.

\begin{question}
In the family of $N$-vertex, $d$-regular graphs $G$, what is the maximum value of $i(G)$?
\end{question}

Alon \cite{Alon91} first conjectured in 1991 that, when $N$ is divisible by $2d$, the maximum should be achieved when $G$ is a disjoint union of complete bipartite graphs $K_{d,d}$. In 2001, Kahn \cite{Kahn} proved Alon's conjecture in the case when $G$ is a bipartite graph. Zhao \cite{Zhao:indep} recently proved the conjecture in general. Theorem \ref{thm:kahn-zhao} contains a precise statement of the result. See \cite{GalDM} or \cite{Zhao:indep} for a history of the problem.

\begin{theorem}[Zhao \cite{Zhao:indep}] \label{thm:kahn-zhao}
For any $N$-vertex, $d$-regular graph $G$,
\[
	i(G) \leq i\(K_{d,d}\)^{N/(2d)} = \(2^{d+1} - 1 \)^{N/(2d)}.
\]
Note that there is equality when $G$ is a disjoint union of $K_{d,d}$'s.
\end{theorem}

This result gives a tight upper bound to the quantity $i(G)^{1/\abs{V(G)}}$ ranged over all $d$-regular graphs $G$. This quantity can be viewed as the number of independent sets normalized by the size of the graph.

In this paper we extend Theorem \ref{thm:kahn-zhao} to give several new results on graph homomorphisms, graph colorings, and stable set polytopes.

For graphs $G$ and $H$ (allowing loops for $H$), a \emph{graph homomorphism} is a essentially a map from the vertices of $G$ to the vertices of $H$ that carries each edge of $G$ to some edge of $H$. More precisely, the set of graph homomorphisms from $G$ to $H$ is given by
\[
	\Hom(G, H) = \{f : V(G) \to V(H) :  f(u)f(v) \in E(H) \; \forall uv \in E(G) \},
\]
($vv$ means a loop at $v$) and let 
\[
	\hom(G, H) = \abs{\Hom(G, H)}.
\]
Graph homomorphisms generalize the notion of independent sets. Indeed if we take $H$ to be the graph with vertices $\{0, 1\}$ and edges $\{00, 01\}$ (see Figure~\ref{fig:H1}), then $\Hom(G, H)$ is in bijection with the collection of independent sets of $G$. Indeed, for each homomorphism from $G$ to $H$, the subset of vertices of $G$ that map to $1 \in V(H)$ forms an independent set. Thus $i(G) = \hom(G, H)$. 

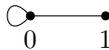
\begin{figure}[!ht]
\centering
\begin{tikzpicture}[sf]
	\draw (1,0) node[P, label=below:$1$] {} -- (0,0) node[P, label=below:$0$] {} edge[-,in = 135, out = 225, loop] ();
\end{tikzpicture}
\caption{A homomorphism from any graph $G$ into the above graph corresponds to an independent set of $G$.\label{fig:H1}}
\end{figure}

In addition, graph homomorphisms generalize proper vertex-colorings. Take $H = K_q$, the complete graph on $q$ vertices. Viewing each vertex of $K_q$ as a color, we see that a homomorphism in $\Hom(G, K_q)$ corresponds to an assignment of each vertex of $G$ to one of $q$ colors so that no two adjacent vertices are assigned the same color. Thus $\hom(G, K_q)$ equals to the number of proper $q$-colorings of $G$.

Since graph homomorphisms generalize independent sets, it is natural to ask whether Theorem \ref{thm:kahn-zhao} can be generalized to graph homomorphisms. Indeed, the following result of Galvin and Tetali \cite{GT} generalizes Theorem \ref{thm:kahn-zhao} in the bipartite case.

\begin{theorem}[Galvin-Tetali \cite{GT}] \label{thm:GT}
For any $N$-vertex, $d$-regular bipartite graph $G$, and any $H$ (possibly with loops), we have
\begin{equation} \label{eq:GT1}
	\hom(G, H) \leq \hom(K_{d,d}, H)^{N/(2d)}.
\end{equation}
\end{theorem}

Note that in contrast to Theorem~\ref{thm:kahn-zhao}, Theorem \ref{thm:GT} requires $G$ to be bipartite. It was conjectured in \cite{GT} that the bipartite condition in Theorem \ref{thm:GT} can be dropped. Unfortunately, this is false for $G = K_3$ and $H$ a graph of two disconnected loops (see Example \ref{ex:GTfail}). We would like to know which graphs $H$ satisfy \eqref{eq:GT1} for all $G$, as it would allow us to address the following question and thereby to generalize Theorem \ref{thm:kahn-zhao} to other instances of graph homomorphisms.

\begin{question}
Let $H$ be a fixed graph (allowing loops). In the family of $N$-vertex, $d$-regular graphs $G$, what is the maximum value of $\hom(G, H)$?
\end{question}

\subsection{Motivation of technique.} \label{sec:intro-motivation}
The proof of Theorem \ref{thm:kahn-zhao} consists of two main steps. The first step, given by Kahn \cite{Kahn}, used entropy methods to prove the theorem when $G$ is bipartite. The second step, given by Zhao~\cite{Zhao:indep}, reduces the general case to the bipartite case through a combinatorial argument. The first step has already been generalized to graph homomorphisms by Galvin and Tetali, resulting in Theorem~\ref{thm:GT}. In this paper we generalize the second step to graph homomorphisms. Since we will be building on the ideas used in the proof of the independent set problem, it will be helpful to recall the argument, as we shall do now.

Let $G \sqcup G$ denote two disjoint copies of $G$, with vertices labeled $v_i$ for $v \in V(G)$ and $i \in \{0,1\}$. Let $G \x K_2$ denote the bipartite graph with vertices also labeled $v_i$ for $v \in V(G)$ and $i \in \{0, 1\}$, but with edges $u_0v_1$ for $uv \in E$. The key step in \cite{Zhao:indep} was to show that $i(G)^2 \leq i(G \x K_2)$. We know that $i(G \x K_2) \leq i(K_{d,d})^{N/d}$ from the bipartite case, so it follows that $i(G)^2 \leq i(G \x K_2) \leq i(K_{d,d})^{N/d}$ and hence $i(G) \leq i(K_{d,d})^{N/(2d)}$.

Note that $i(G)^2 = i(G \sqcup G)$. The proof of the inequality $i(G)^2 \leq i(G \x K_2)$ involves constructing an injection from the collection of independent sets of $G \sqcup G$ to that of $G \x K_2$. A snapshot of this construction is illustrated in Figure~\ref{fig:indep-swap}. We start from an independent set of $G \sqcup G$ (the black vertices in the figure). After ``crossing'' the edges to transform $G \sqcup G$ into $G \x K_2$, we get a subset of the vertices of $G \x K_2$ (middle figure) which might not be an independent set in $G \x K_2$. However, it turns out that we can always ``swap'' a number of pairs of vertices (each pair is shown in a dashed circle) so that the resulting subset of vertices is an independent set in $G \x K_2$. It takes a bit of thought to see that such swapping is always possible. It is true  because the set of ``bad'' edges in $G$, corresponding to those edges in $G \x K_2$ whose both endpoints are selected, form a bipartite subgraph of $G$. Once we specify a uniform way of choosing of the set of vertices to swap---one recipe is to always choose the lexicographically first subset of $V(G)$ that ``works''---we will have a method of transforming an independent set of $G \sqcup G$ into an independent set of $G \x K_2$. This map is injective as long as there is a way of recovering the set of swapped pairs of vertices---if we had chosen the lexicographically first subset of vertices to swap, then we can recover our choice by choosing the lexicographically first subset of $V(G)$ whose swapping gives an independent set of $G \sqcup G$ after ``uncrossing'' the edges of $G \x K_2$ to get $G \sqcup G$. This completes the proof that $i(G)^2 = i(G \sqcup G) \leq i(G \x K_2)$.

\begin{figure}[!ht]
\centering
\tikzstyle{B}=[draw,circle, fill=black, minimum size=4pt,inner sep=0pt]
\tikzstyle{W}=[draw,circle, fill=white, minimum size=4pt,inner sep=0pt]

\begin{tikzpicture}[scale=.8, sf]
\begin{scope}[shift={(0,0)}]
	\begin{scope}[shift={(-.1,.2)}]
		\node[W] (a) at (0,0) {};
		\node[B] (b) at (3,0) {};
		\node[W] (c) at (1,1) {};
		\node[W] (d) at (2,1) {};
		\node[B] (e) at (0,2) {};
		\node[W] (f) at (3,2) {};
	\end{scope}
	\begin{scope}[shift={(.1,-.2)}]
		\node[W] (a1) at (0,0) {};
		\node[W] (b1) at (3,0) {};
		\node[B] (c1) at (1,1) {};
		\node[W] (d1) at (2,1) {};
		\node[W] (e1) at (0,2) {};
		\node[B] (f1) at (3,2) {};
	\end{scope}	
	\draw (a)--(b)--(d)--(c)--(a)--(e)--(c); \draw (e)--(f)--(b); \draw (d)--(f);	
	\draw (a1)--(b1)--(d1)--(c1)--(a1)--(e1)--(c1); \draw (e1)--(f1)--(b1); \draw (d1)--(f1);
	
	\draw[-latex] (3.5,1)--(4.5,1);
	\node at (1.5,-.8) {$G\sqcup G$};
\end{scope}

\begin{scope}[shift={(5,0)}]
	\begin{scope}[shift={(-.1,.2)}]
		\node[W] (a) at (0,0) {};
		\node[B] (b) at (3,0) {};
		\node[W] (c) at (1,1) {};
		\node[W] (d) at (2,1) {};
		\node[B] (e) at (0,2) {};
		\node[W] (f) at (3,2) {};
	\end{scope}
	\begin{scope}[shift={(.1,-.2)}]
		\node[W] (a1) at (0,0) {};
		\node[W] (b1) at (3,0) {};
		\node[B] (c1) at (1,1) {};
		\node[W] (d1) at (2,1) {};
		\node[W] (e1) at (0,2) {};
		\node[B] (f1) at (3,2) {};
	\end{scope}
	\draw[dashed] (0,2) circle (.8);
	\draw[dashed] (3,0) circle (.8);
	\draw (a)--(b1)--(d)--(c1)--(a)--(e1)--(c); \draw (e1)--(f)--(b1); \draw (d1)--(f);	
	\draw (a1)--(b)--(d1)--(c)--(a1)--(e)--(c1); \draw (e)--(f1)--(b); \draw (d)--(f1);	
	
	\draw[-latex] (3.5,1)--(4.5,1);
	\node at (1.5,-.8) {$G \x K_2$};
\end{scope}

\begin{scope}[shift={(10,0)}]
	\begin{scope}[shift={(-.1,.2)}]
		\node[W] (a) at (0,0) {};
		\node[W] (b) at (3,0) {};
		\node[W] (c) at (1,1) {};
		\node[W] (d) at (2,1) {};
		\node[W] (e) at (0,2) {};
		\node[W] (f) at (3,2) {};
	\end{scope}
	\begin{scope}[shift={(.1,-.2)}]
		\node[W] (a1) at (0,0) {};
		\node[B] (b1) at (3,0) {};
		\node[B] (c1) at (1,1) {};
		\node[W] (d1) at (2,1) {};
		\node[B] (e1) at (0,2) {};
		\node[B] (f1) at (3,2) {};
	\end{scope}
	\draw (a)--(b1)--(d)--(c1)--(a)--(e1)--(c); \draw (e1)--(f)--(b1); \draw (d1)--(f);	
	\draw (a1)--(b)--(d1)--(c)--(a1)--(e)--(c1); \draw (e)--(f1)--(b); \draw (d)--(f1);	
	\node at (1.5,-.8) {$G \x K_2$};
\end{scope}
\end{tikzpicture}
\caption{Transforming an independent set of $G \sqcup G$ to an independent set of $G \x K_2$. \label{fig:indep-swap}}
\end{figure}
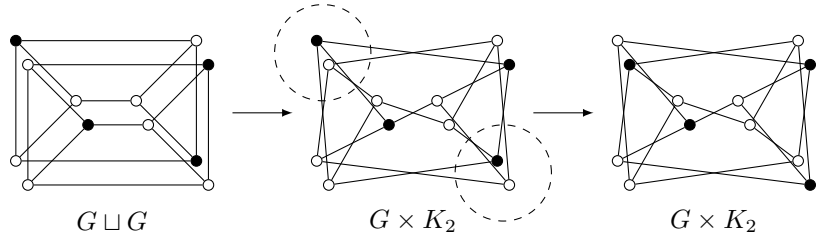

We would like to extend the comparison between $G \sqcup G$ and $G \x K_2$ from independent sets to graph homomorphisms. We introduce the \emph{bipartite swapping trick} (Proposition~\ref{prop:bst}, which is a generalize of the above injection. The bipartite swapping trick gives us a method of corresponding certain elements of $\Hom(G \sqcup G, H)$ with those of $\Hom(G \x K_2, H)$. For instance, when $H$ is a \emph{bipartite swapping target} (Definition~\ref{def:target}), there is an injection from $\Hom(G \sqcup G, H)$ to $\Hom(G \x K_2, H)$, thereby allowing us to extend Theorem \ref{thm:GT} to non-bipartite $G$ in certain cases.

\medskip

\noindent \emph{Outline of paper.} In Section~\ref{sec:summary}, we give a summary of our results and introduce the notion of GT graphs, which characterizes when Theorem~\ref{thm:GT} can be extended to non-bipartite graphs. In Section~\ref{sec:trick} we describe the bipartite swapping trick. In Section~\ref{sec:target} we consider families of graphs where the bipartite swapping trick always succeeds in proving the non-bipartite extension of Theorem~\ref{thm:GT}. In Section~\ref{sec:coloring} we apply the bipartite swapping trick to counting graph colorings. In Section~\ref{sec:polytope} we apply our results to the stable set polytope of a graph. Finally, in Section~\ref{sec:weighted}, we consider weighted generalizations of our results. Although the proofs of the weighted analogs of our results come at almost no extra effort, we choose to defer the discussion until the end in order to simplify the presentation.

\noindent \emph{Notation and convention.}
In this paper, $G$ always denotes the source of a graph homomorphism and $H$ always denotes the target. All graphs are undirected. The graph $G$ is simple. We allow loops for $H$ but not parallel edges or parallel loops. The notations $V(\cdot)$ and $E(\cdot)$ respectively denote the set of vertices and the set of edges of a graph. The function $\Hom( \cdot, \cdot)$ (and its variants) always returns a set while $\hom( \cdot, \cdot)$ always returns a number.

\section{Statement of results} \label{sec:summary}

\subsection{\boldmath{$GT$} graphs} \label{sec:summary-GT}

As motivated in the introduction, we are interested in extending Theorem \ref{thm:GT} to non-bipartite $G$.

\begin{definition} \label{def:GT}
A graph $H$ (not necessarily simple) is \emph{GT} if 
\begin{equation}\label{eq:def:GT}
	\hom(G, H) \leq \hom(K_{d,d}, H)^{N/(2d)}
\end{equation}
holds for every $N$-vertex, $d$-regular graph $G$.
\end{definition}

\begin{example}
The graph $H$ in Figure~\ref{fig:H1} is GT by Theorem \ref{thm:kahn-zhao} since $\hom(G, H) = i(G)$ for every $G$. 
\end{example}

\begin{example} \label{ex:GTfail}
Let $H$ be the graph with two disconnected vertices, each with a loop. Then $H$ is not GT. Indeed, let $G = K_3$. Then $\hom(G, H) = 2 >  2^{3/4} = \hom(K_{2,2}, H)^{3/4}$.
\end{example}

Theorem \ref{thm:GT} implies that \eqref{eq:def:GT} is true for bipartite $G$. As motivated in Section~\ref{sec:intro-motivation}, we would like to reduce the general case to the bipartite case by comparing $G \sqcup G$ and $G \x K_2$.

\begin{definition} \label{def:sGT}
A graph $H$ (not necessarily simple) is \emph{strongly GT} if
\[
	\hom(G \sqcup G, H) \leq \hom(G \x K_2, H)
\]
for every graph $G$ (not necessarily regular).
\end{definition}

The following lemma shows the significance of being strongly GT and also justifies the terminology.

\begin{lemma} \label{lem:sGT-GT}
If $H$ is strongly GT, then it is GT.
\end{lemma}

\begin{proof}
Suppose $H$ is strongly GT. Let $G$ be an $N$-vertex, $d$-regular graph. Note that $G\x K_2$ is a $2N$-vertex, $d$-regular bipartite graph, so we may apply Theorem \ref{thm:GT}. Then $H$ being strongly $GT$ implies that
\[
	\hom(G, H)^2 = \hom(G \sqcup G, H) \leq \hom(G \x K_2, H) \leq \hom(K_{d,d}, H)^{N/d}.
\]
Therefore $\hom(G, H) \leq \hom(K_{d,d}, H)^{N/(2d)}$, and hence $H$ is GT.
\end{proof}

\begin{example}
The argument in Section~\ref{sec:intro-motivation}, originally from \cite{Zhao:indep}, shows that $i(G \sqcup G) \leq i(G \x K_2)$ for all graphs $G$, so the graph $H$ in Figure \ref{fig:H1} is strongly GT.
\end{example}

\begin{remark}
If $G$ is bipartite, then the graphs $G \sqcup G$ and $G \x K_2$ are isomorphic. Indeed, if $V(G) = A \sqcup B$ is a bipartition, then the map $V(G \sqcup G) \to V(G \x K_2)$ sending $v_i$ to $v_i$ if $v \in A$ and $v_{1-i}$ if $v \in B$ gives a graph isomorphism.

If $H$ is bipartite, then $\Hom(G, H) = \emptyset$ unless $G$ is bipartite, so $\hom(G, H)^2 = \hom(G \sqcup G, H) = \hom(G \x K_2, H)$ if $G$ is bipartite and $\hom(G, H) = 0$ otherwise. Thus every bipartite graph is strongly GT in an uninteresting way.

We suspect that there exists graphs which are GT but not strongly GT. Unfortunately, we do not know any examples.
\end{remark}

In this paper, we provide some sufficient conditions for a graph to be GT. Here is a road map for our chain of implications.
\begin{align*}
	&H \text{ is a threshold graph (Definition~\ref{def:H}, Theorem~\ref{thm:4c})}
\\	\overset{\text{Prop.~\ref{prop:4c-bst}}}{\Longrightarrow} \quad & H \text{ is a bipartite swapping target (Definition~\ref{def:target})}
\\	\overset{\text{Cor.~\ref{cor:target-ineq}}}{\Longrightarrow} \quad & H \text{ is strongly GT (Definition~\ref{def:sGT})}
\\	\overset{\text{Lem.~\ref{lem:sGT-GT}}}{\Longrightarrow} \quad & H \text{ is GT (Definition~\ref{def:GT}).}	
\end{align*}
Threshold graphs are graphs with vertices are a multiset of real numbers, and an edge between two vertices whenever their sum does not exceed a certain global threshold. The graph in Figure~\ref{fig:H1} is an example of a threshold graph, so our new result generalizes Theorem~\ref{thm:kahn-zhao}. We also provide weighted generalizations in Section \ref{sec:weighted}.

\subsection{Counting graph colorings} \label{sec:summary-coloring}

The case $H = K_q$ is particularly significant, since $\Hom(G, K_q)$ is in bijection with the set of all proper $q$-colorings of $G$, i.e., ways of coloring the vertices of $G$ using at most $q$ colors so that no two adjacent vertices are assigned the same color. The function
\[
P(G,q) = \hom(G,K_q)
\]
is known as the chromatic polynomial of $G$ (viewed as a function in $q$) and it counts the number of proper $q$-colorings of $G$. The problem of maximizing/minimizing the number of $q$-colorings over various families of graph has been intensely studied, especially the family of graphs with a fixed number of vertices and edges. See the introduction of \cite{LPS} for an overview of the state of this problem. Here we are interested in maximizing the number of $q$-colorings in the family of $N$-vertex, $d$ regular graphs.

\begin{conjecture} \label{conj:coloring}
For $q \geq 3$, the complete graph $K_q$ is GT. Equivalently, for any $N$-vertex, $d$-regular graph $G$, the chromatic polynomial satisfies
\begin{equation} \label{eq:coloring-ineq}
	P(G, q) \leq P( K_{d,d}, q)^{N/(2d)},
\end{equation}
Note that we have equality when $G$ is sa disjoint union of $K_{d,d}$'s.
\end{conjecture}

From Theorem~\ref{thm:GT}, we know that Conjecture~\ref{conj:coloring} is true when $G$ is bipartite. Although we do not know how to prove the conjecture, we can show the following asymptotic result using our bipartite swapping trick.

\begin{theorem} \label{thm:asym-coloring}
For every $N$-vertex, $d$-regular graph $G$,
\[
	P(G, q) \leq P( K_{d,d}, q)^{N/(2d)}
\]
for all sufficiently large $q$ (depending on $N$). Note that equality occurs when $G$ is a disjoint union of $K_{d,d}$'s.
\end{theorem}

\subsection{Generalized independent sets} \label{summary-gen-indep}

Let $\c I(G, n)$ denote the set of assignments $f : V \to \{0, 1, \dots, n\}$ so that the sum of the endpoints of an edge never exceeds $n$. Let $i(G, n) = \abs{\c I(G, n)}$. When $n = 1$, this construction corresponds to independent sets, so the following result is a generalization of Theorem \ref{thm:kahn-zhao}.

\begin{theorem} \label{thm:i}
For any $N$-vertex, $d$-regular graph $G$, and positive integer $n$,
\[
	i(G, n) \leq i(K_{d,d}, n)^{N/(2d)}.
\]
\end{theorem}

The collection $\c I(G, n)$ arises naturally in statistical mechanics \cite{MazelSuhov} and communication networks \cite{GMRT, RSZM}. Galvin et al.~\cite{GMRT} related it to the ``finite-state hard core model.'' In these stochastic modeling applications, it is common to weight each assignment in $\c I(G,n)$ using a geometric or Poisson distribution. Our results also extend to weighted generalization, which are discussed in Section~\ref{sec:weighted}. In fact, Theorem \ref{thm:i} remains true if we replace $\c I(G, n)$ by the collections of assignments $f : V \to A$, where $A$ is some fixed finite set of real numbers, so that the sum of the numbers assigned to endpoints of an edge never exceed some threshold.

\subsection{Stable set polytope} \label{sec:summary-vol}

We consider one more measure on the independent sets of $G$, namely the volume $i_V(G)$ of the \emph{stable set polytope} of $G$, which is defined to be the convex hull of the characteristic vectors of the independent sets of $G$. We prove the following inequality, which has a form analogous to the previous results.

\begin{theorem} \label{thm:iV}
For any $N$-vertex, $d$-regular graph $G$, the volume of the stable set polytope of $G$ satisfies
\[
	i_V(G) \leq i_V\(K_{d,d}\)^{N/(2d)} = \binom{2d}{d}^{-N/(2d)}.
\]
\end{theorem}


\section{Bipartite swapping trick} \label{sec:trick}

In this section, we describe the main technique of our paper. Our goal is to construct a correspondence between a subset of $\Hom(G \sqcup G, H)$ and a subset of $\Hom(G \x K_2, H)$.

We name the vertices of both $G \sqcup G$ and $G \x K_2$ by $v_i$, for $v \in V$ and $i \in \{0, 1\}$, such that the edges in $G \sqcup G$ are $u_iv_i$ and edges of $G \x K_2$ are $u_iv_{1-i}$, for $uv \in E$ and $i \in \{0, 1\}$. 



Let us describe a representation of elements of $\Hom(G \sqcup G, H)$ and $\Hom(G \x K_2, H)$. An \emph{$H$-pair-labeling} of $G$ is simply an assignment $V(G) \to V(H) \x V(H)$, with no additional constraints. Equivalently, it is a way of labeling each vertex of $G$ with a pair of vertices of $H$. Every $f \in \Hom(G \cup G, H)$ can be represented by an $H$-pair-labeling of $G$ with additonal constraints, assigning $v \in V(G)$ to the pair $(f(v_0), f(v_1))$, satisying the constraints that whenever $uv \in E(G)$, the first vertex of $H$ assigned to $u$ must be adjacent (in $H$) to the first vertex assigned to $v$, and the second vertex assigned to $u$ must be adjacent to the second vertex assigned to $v$. It is easy to see that this describes a bijective correspondence between $\Hom(G \cup G, H)$ and the set of $H$-pair-labelings satisfying these constraints. Similarly, we can represent elements of $\Hom(G \x K_2, H)$ by $H$-pair-labelings satisfying the constraint that whenever $uv \in E(G)$, the first vertex assigned to $u$ must be adjacent to the second vertex assigned to $v$.

If $f \in \Hom(G \cup G, H)$ or $f \in \Hom(G \x K_2, H)$, we denote by $\pair(f)$ the corresponding $H$-pair-labeling.


\begin{example}
Let $G$ and $H$ be the following graphs. The vertices of $H$ are named $a, b, c, d$.
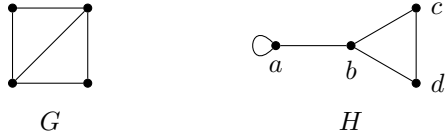
\begin{center}
\begin{tikzpicture}[sf]
	\begin{scope}[xshift=4cm]
		\node[P, label=below:$a$] (a) at (180:1) {};
		\node[P, label=below:$b$] (b) at (0,0) {};
		\node[P, label=right:$c$] (c) at (30:1) {};
		\node[P, label=right:$d$] (d) at (-30:1) {};
		\draw (a) edge[-,in = 135, out = 225, loop] (a) -- (b) -- (c) -- (d) -- (b);
		\node at (0,-1) {$H$};
	\end{scope}
	
	\begin{scope}
		\node[P] (1) at (.5,.5) {};
		\node[P] (2) at (-.5,.5) {};
		\node[P] (3) at (-.5,-.5) {};
		\node[P] (4) at (.5,-.5) {};
		\draw (1)--(2)--(3)--(4)--(1)--(3);
		
		\node at (0,-1) {$G$};
	\end{scope}
\end{tikzpicture}
\end{center}
Then the $H$-pair-labeling on the left diagram below represents an element of $\Hom(G \sqcup G, H)$ (but not an element of $\Hom(G \x K_2, H)$), while the $H$-pair-labeling on the right diagram below represents an element of $\Hom(G \x K_2, H)$ (but not an element of $\Hom(G \sqcup G, H)$). Recall that in both cases we label each $v \in V(G)$ by $(f(v_0), f(v_1))$.
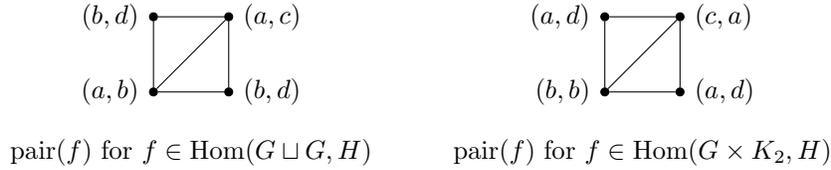
\begin{center}
\begin{tikzpicture}[sf]
	\begin{scope}
		\node[P, label=right:{$(a,c)$}] (1) at (.5,.5) {};
		\node[P, label=left:{$(b,d)$}] (2) at (-.5,.5) {};
		\node[P, label=left:{$(a,b)$}] (3) at (-.5,-.5) {};
		\node[P, label=right:{$(b,d)$}] (4) at (.5,-.5) {};
		\draw (1)--(2)--(3)--(4)--(1)--(3);
		
		\node at (0,-1.3) {\small $\pair(f)$ for $f \in \Hom(G \sqcup G, H)$};
	\end{scope}
	\begin{scope}[xshift=6cm]
		\node[P, label=right:{$(c,a)$}] (1) at (.5,.5) {};
		\node[P, label=left:{$(a,d)$}] (2) at (-.5,.5) {};
		\node[P, label=left:{$(b,b)$}] (3) at (-.5,-.5) {};
		\node[P, label=right:{$(a,d)$}] (4) at (.5,-.5) {};
		\draw (1)--(2)--(3)--(4)--(1)--(3);
		
		\node at (0,-1.3) {\small $\pair(f)$ for $f \in \Hom(G \x K_2, H)$};
	\end{scope}
\end{tikzpicture}
\end{center}

\end{example}

We wish to transform a homomorphism $f \in \Hom(G \sqcup G, H)$ into a homomorphism in $\Hom(G \x K_2, H)$. We might naively do by hoping that the same map of vertices works, that is, perhaps we can keep the same $H$-pair-labeling representation. However, this does not always work, because the same $H$-pair-labeling might no longer represent a homomorphism in $\Hom(G \x K_2, H)$, as is the case in the previous example. The problem is that the $H$-pair-labeling needs to satisfy different contraints to be a homomorphism in $\Hom(G \sqcup G, H)$ and in $\Hom(G \x K_2, H)$. The following definition is motivated by this obstruction.

\begin{definition} \label{def:violated}
Let $p = (p_1,p_2) : V(G) \to V(H) \x V(H)$ be an $H$-pair-labeling of $G$. We say that $uv \in E(G)$ is \emph{safe} with respect to $p$ if $p_i(u)p_j(v) \in E(H)$ for all $i,j \in \set{0,1}$, otherwise we say that $uv$ is \emph{violated} with respect to $p$.

If $f \in \Hom(G \sqcup G, H)$ or $f \in \Hom(G \x K_2, H)$, then we say that $uv \in E(G)$ is safe (resp.~violated with respect to $f$ if the corresponding $H$-pair-labeling is safe (resp.~violated) with respect to $\pair(f)$.
\end{definition}


Note that we speak of edges of $G$ being violated, and not edges of $G \sqcup G$ or $G \x K_2$. For instance, when we say that $uv \in E(G)$ is violated with respect to $f \in \Hom(G \sqcup G, H)$, the violation refers to not what happens in the current homomorphism (as $f$ is a valid homomorphism), but the obstructions to a homorphism once we transform $G \sqcup G$ to $G \x K_2$ and attempting to keep the ``same'' $f$.

\begin{example} \label{ex:map-assign}
A homomorphism to $K_q$ is the same as a proper $q$-coloring of the graph. Suppose that we represent the colors (i.e. vertices of $K_q$) by letters. In the diagrams below, the first $K_q$-pair-labeling on the left represents an element of $\Hom(G \sqcup G, K_4)$ and the second $K_q$-pair-labeling represents an element of $\Hom(G \x K_2, K_4)$. The violated edges of $G$ in each case is highlighted in bold.
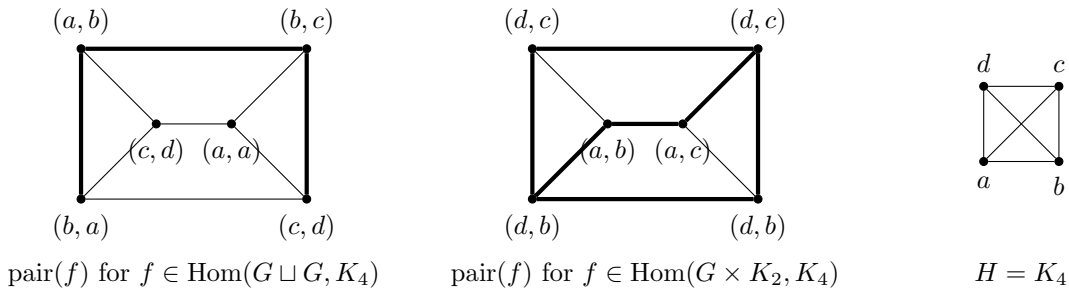
\begin{center}
\begin{tikzpicture}[sf]
	\begin{scope}
		\node[P, label=below:{$(b,a)$}] (1) at (0,0) {};
		\node[P, label=below:{$(c,d)$}] (2) at (3,0) {};
		\node[P, label=below:{$(c,d)$}] (3) at (1,1) {};
		\node[P, label=below:{$(a,a)$}] (4) at (2,1) {};
		\node[P, label=above:{$(a,b)$}] (5) at (0,2) {};
		\node[P, label=above:{$(b,c)$}] (6) at (3,2) {};
		\draw (1)--(2)--(4)--(3)--(1)--(5)--(3); \draw (5)--(6)--(2); \draw (4)--(6);	
		\node at (1.5,-1) {$\pair(f)$ for $f \in \Hom(G \sqcup G, K_4)$};
		\draw[t] (1)--(5)--(6)--(2);
	\end{scope}
	
	\begin{scope}[xshift=6cm]
		\node[P, label=below:{$(d,b)$}] (1) at (0,0) {};
		\node[P, label=below:{$(d,b)$}] (2) at (3,0) {};
		\node[P, label=below:{$(a,b)$}] (3) at (1,1) {};
		\node[P, label=below:{$(a,c)$}] (4) at (2,1) {};
		\node[P, label=above:{$(d,c)$}] (5) at (0,2) {};
		\node[P, label=above:{$(d,c)$}] (6) at (3,2) {};
		\draw (1)--(2)--(4)--(3)--(1)--(5)--(3); \draw (5)--(6)--(2); \draw (4)--(6);	
		\node at (1.5,-1) {$\pair(f)$ for $f \in \Hom(G \x K_2, K_4)$};
		\draw[t] (1)--(2)--(6)--(5)--(1)--(3)--(4)--(6);
	\end{scope}
	
	\begin{scope}[xshift=12cm]
		\node[P, label=below:$a$] (a) at (0,0.5) {};
		\node[P, label=below:$b$] (b) at (1,0.5) {};
		\node[P, label=above:$c$] (c) at (1,1.5) {};
		\node[P, label=above:$d$] (d) at (0,1.5) {};
		\draw (a)--(b)--(c)--(d)--(a)--(c);
		\draw (b)--(d);
		\node at (.5, -1) {$H = K_4$};
	\end{scope}
\end{tikzpicture}
\end{center}
\end{example}

Here is the key operation used in the bipartite swapping trick.

\begin{definition}[The swapping operation] \label{def:swap}
  Let $p$ be an $H$-pair-labeling of $G$, and let $W \subseteq V(G)$. Define $\swap(p, W)$ to be the $H$-pair-labeling obtained from $p$ by swapping each pair of labels assigned to vertices in $W$.
\end{definition}

Note that swapping does not affect whether an edge is violated.

The key insight is that if we start with $f \in \Hom(G \sqcup G, H)$, then the violated edges prevent $f$ from being a valid homomorphism in $\Hom(G \x K_2, H)$, but we can fix this issue by  swapping exactly one endpoint of each violated edge. In order to perform this operation successfully to the whole graph, the set of violated edges must form a bipartite subgraph, hence the following definition.

\begin{definition}[Bipartite swapping property] \label{def:bsp}
Let $p$ be an $H$-pair-labeling of $G$. We say that $p$ has the \emph{bipartite swapping property} if the edges of $G$ that are violated with respect to $p$ is a bipartite subgraph of $G$. Similarly, we say that $f \in \Hom(G \sqcup G, H)$ or $f \in \Hom(G \x K_2, H)$ has the bipartite swapping property if $\pair(f)$ does.
\end{definition}



Note that bipartite-ness appears in two separate places. The first is where we compare any arbitrary graph $G$ to a bipartite graph $G \x K_2$. The second is where we consider bipartite subgraphs of $G$.

\begin{example} \label{ex:swap}
In Example~\ref{ex:map-assign}, the first homomorphism has the bipartite swapping property while the second one does not. Let $f$ denote the first homomorphism, whose $H$-pair-labeling is reproduced below on the left. Let $W$ denote the set of circled vertices. Then $\swap(\pair(f), W)$, shown on the right, represents an element of $\Hom(G \x K_2, K_4)$. Note that $W$ contains exactly one endpoint of every violated edge.
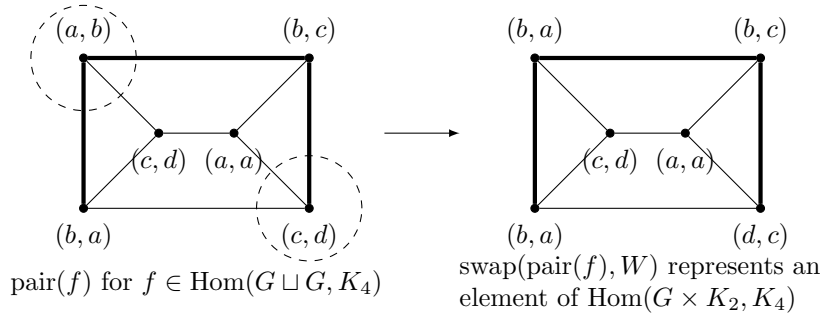
\begin{center}
\begin{tikzpicture}[sf]
	\begin{scope}
		\node[P, label=below:{$(b,a)$}] (1) at (0,0) {};
		\node[P, label=below:{$(c,d)$}] (2) at (3,0) {};
		\node[P, label=below:{$(c,d)$}] (3) at (1,1) {};
		\node[P, label=below:{$(a,a)$}] (4) at (2,1) {};
		\node[P, label=above:{$(a,b)$}] (5) at (0,2) {};
		\node[P, label=above:{$(b,c)$}] (6) at (3,2) {};
		\draw (1)--(2)--(4)--(3)--(1)--(5)--(3); \draw (5)--(6)--(2); \draw (4)--(6);	
		\node at (1.5,-1) {$\pair(f)$ for $f \in \Hom(G \sqcup G, K_4)$};
		\draw[t] (1)--(5)--(6)--(2);
		\draw[dashed] (2) circle (.7);
		\draw[dashed] (5) circle (.7);
		
		\draw[-latex] (4,1)--(5,1);
	\end{scope}
	
	\begin{scope}[xshift=6cm]
		\node[P, label=below:{$(b,a)$}] (1) at (0,0) {};
		\node[P, label=below:{$(d,c)$}] (2) at (3,0) {};
		\node[P, label=below:{$(c,d)$}] (3) at (1,1) {};
		\node[P, label=below:{$(a,a)$}] (4) at (2,1) {};
		\node[P, label=above:{$(b,a)$}] (5) at (0,2) {};
		\node[P, label=above:{$(b,c)$}] (6) at (3,2) {};
		\draw (1)--(2)--(4)--(3)--(1)--(5)--(3); \draw (5)--(6)--(2); \draw (4)--(6);	
		\node[text width=5cm] at (1.5,-1) { $\swap(\pair(f), W)$ represents an element of $\Hom(G \x K_2, K_4)$};
		\draw[t] (1)--(5)--(6)--(2);
	\end{scope}
	
\end{tikzpicture}
\end{center}
\end{example}

Let $\Hom^{\bsp}(G \sqcup G, H)$ denote the subset of $\Hom(G \sqcup G, H)$ containing all homomorphisms possessing the bipartite swapping property. Similarly let $\Hom^{\bsp}(G \x K_2, H)$ denote the subset of $\Hom(G \x K_2, H)$ containing all homomorphisms possessing the bipartite swapping property.

\begin{proposition}[Bipartite swapping trick] \label{prop:bst}
For graphs $G$ and $H$ ($H$ possibly with loops), there exists a bijection between $\Hom^{\bsp}(G \sqcup G, H)$ and $\Hom^{\bsp}(G \x K_2, H)$, obtained through some application of the swapping operation.
\end{proposition}

We need to address two issues. First we need to check that such swapping operation produces valid homomorphisms. Second we need to describe how to consistently choose the subset of vertices of $G$ in order to make the map a bijection.

\begin{lemma} \label{lem:swap}
Suppose $f \in \Hom^\bsp(G \sqcup G, H)$ (resp.~$\Hom^\bsp(G \x K_2, H)$) and $W \subseteq V(G)$, such that each violated edge with respect to $f$ has exactly one endpoint in $W$. Then $\swap(\pair(f), W)$ represents an element of $\Hom^\bsp(G \x K_2, H)$ (resp.~$\Hom^\bsp(G \sqcup G, H)$).
\end{lemma}

\begin{proof}
We check the $f \in \Hom^\bsp(G \sqcup G, H)$ case (the other one is analogous). Let $p = \pair(f)$ and $p' = \swap(p, W)$, so that $p'_i(v) = p_i(v)$ for all $v \notin W$ and $i \in \set{0,1}$, and $p'_i(v) = p'_{1-i}(v)$ for $v \in W$ and $i \in \set{0,1}$. We want to show that $p'$ represents an element of $\Hom^\bsp(G \x K_2, H)$. So we need to check that if $uv \in E(G)$, then $p'_i(u)p'_{1-i}(v) \in E(H)$ for $i \in \set{0,1}$. If $uv$ is safe with respect to $p$ (and hence $p'$ as well), then we automatically have $p'_i(u)p'_{1-i}(v) \in E(H)$. Otherwise, $uv$ is vioated, so exactly one of $u$ and $v$ is contained in $W$. Say $v \in W$. Then $p'_i(u)p'_{1-i}(v) = p_i(u)p_i(v)$, which is in $E(H)$ since $p$ represents an element of $\Hom(G \sqcup G, H)$. It follows that $p'$ represents an element of $\Hom(G \x K_2, H)$. Note that the set of violated edges is not affected by swapping, so $p'$ also has the bipartite swapping property, and hence represents an element of $\Hom^\bsp(G \x K_2, H)$.\end{proof}

Given $f \in \Hom^\bsp(G \sqcup G, H)$, the set of violated edges form a bipartite graph, but since there is no canonical bipartition, there may be many choices for $W$ as in the lemma. How do we consistently choose $W$ so that we have a bijection? The rest of the proof address this question.

\begin{proof}[Proof of the Proposition~\ref{prop:bst}]
For every $F \subseteq E(G)$ that forms a bipartite subgraph of $G$, choose $W_F \subseteq V(G)$ so that every edge in $F$ has exactly one endpoint in $W_F$. The specific choice of $W_F$ is unimportant; it just needs to be chosen once and for all.

Construct a bijection between $\Hom^\bsp(G \sqcup G, H)$ and $\Hom^\bsp(G \sqcup G, H)$ by sending $f\in \Hom^\bsp(G \sqcup G, H)$ to the element of $\Hom^\bsp(G \x K_2, H)$ represented by $\swap(\pair(f), W_{\viol(f)})$, where $\viol(f)$ denotes the set of violated edges of $f$. Lemma~\ref{lem:swap} guarantees that the image lends in $\Hom^\bsp(G \x K_2, H)$. For the inverse map, we note that the set of violated edges does not change, so that we can send $f' \in \Hom^\bsp(G \x K_2, H)$ to the element of $\Hom^\bsp(G \sqcup G, H)$ represented by $\swap(\pair(f'), W_{\viol(f')})$. This gives a bijection.
\end{proof}

\section{Bipartite swapping target} \label{sec:target}

In the previous section we saw that there exists a bijective correspondence between $\Hom^{\bsp}(G \sqcup G, H)$ and $\Hom^{\bsp}(G \x K_2, H)$. Sometimes it happens that every homomorphism in $\Hom(G \sqcup G, H)$ has the bipartite swapping property, and in this section we study such cases.

\subsection{Bipartite swapping target}

\begin{definition}[Bipartite swapping target] \label{def:target}
We say that a graph $H$ (not necessarily simple) is a \emph{bipartite swapping target} if $\Hom^{\bsp}(G \sqcup G, H) = \Hom(G \sqcup G, H)$ for every graph $G$, i.e., every homomorphism from $G \sqcup G$ to $H$ has the bipartite swapping property.
\end{definition}

\begin{remark}
If $H$ is a bipartite swapping target, then any induced subgraph of $H$ is also a bipartite swapping target. In other words, being a bipartite swapping target is a hereditary property.
\end{remark}

\begin{example}
Every bipartite graph $H$ is a bipartite swapping target, since $\Hom(G \sqcup G, H) = \emptyset$ unless $G$ is already bipartite.
\end{example}

\begin{example} \label{ex:odd-cycle-not-target}
An odd cycle $H = C_n$ is not a bipartite swapping target. Indeed, if the vertices of $C_n$ are given by elements of $\ZZ/n\ZZ$, with edges between $i$ and $i+1$, then the $H$-pair-labeling $i \mapsto (i, i+1)$ on $G = C_n$ represents an element of $\Hom(C_n \sqcup C_n, C_n)$ that has every edge of $G$ violated. The following diagram shows the example of a 5-cycle.
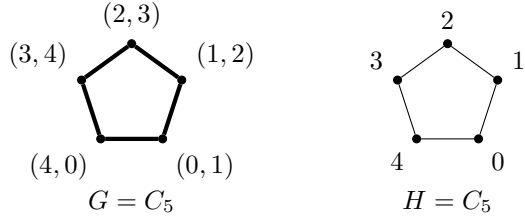
\begin{center}
\begin{tikzpicture}[sf,scale=.7]
	\begin{scope}
		\node[P, label=-54:{$(0,1)$}](0) at (-54:1) {};
		\node[P, label=18:{$(1,2)$}] (1) at (18:1) {};
		\node[P, label=90:{$(2,3)$}] (2) at (90:1) {};
		\node[P, label=162:{$(3,4)$}] (3) at (162:1) {};
		\node[P, label=234:{$(4,0)$}] (4) at (234:1) {};
		\draw[t] (0)--(1)--(2)--(3)--(4)--(0);
		\node at (0,-2) {$G = C_5$};
	\end{scope}
	
	\begin{scope}[xshift=6cm]
		\node[P, label=-54:{$0$}] (0) at (-54:1) {};
		\node[P, label=18:{$1$}] (1) at (18:1) {};
		\node[P, label=90:{$2$}] (2) at (90:1) {};
		\node[P, label=162:{$3$}] (3) at (162:1) {};
		\node[P, label=234:{$4$}] (4) at (234:1) {};
		\draw(0)--(1)--(2)--(3)--(4)--(0);
		\node at (0,-2) {$H = C_5$};
	\end{scope}
\end{tikzpicture}
\end{center}
\end{example}

\begin{proposition} \label{prop:injection}
If $H$ is a bipartite swapping target, then the bipartite swapping trick gives an injective map from $\Hom(G \sqcup G, H)$ to $\Hom(G \x K_2, H)$.
\end{proposition}

\begin{proof}
We have $\Hom(G \sqcup G, H) = \Hom^\bsp(G \sqcup G, H)$, which, by the bipartite swapping trick, is in bijective correspondence with $\Hom^\bsp(G \x K_2, H)$, which is a subset of of $\Hom(G \x K_2, H)$.
\end{proof}

\begin{corollary} \label{cor:target-ineq}
If $H$ is a bipartite swapping target, then $H$ is strongly GT (Definition~\ref{def:sGT}), and hence $H$ is GT (Definition~\ref{def:GT}).
\end{corollary}

\begin{proof}
Proposition \ref{prop:injection} implies that 
\[
\hom(G \sqcup G, H) \leq \hom(G \x K_2, H),
\]
for every $G$. Therefore $H$ is strongly GT.
\end{proof}

\subsection{Testing for bipartite swapping targets}
From Definition \ref{def:target} it seems that to determine whether $H$ is a bipartite swapping target, we have to check the condition for every $G$ and every homomorphism. Fortunately, there is an easy criterion for determining whether a graph is a bipartite swapping target which involves checking whether a particular subgraph of $H \x H$ is bipartite, as we shall explain in this section.

Construct the graph $H^\bst$ with vertices $V(H^\bst) = V(H) \x V(H)$, and an edge between $(u,v)$ and $(u',v') \in V(H^\bst)$ if and only if
\[
	uu' \in E(H), \quad \text{and} \quad
	vv' \in E(H), \quad \text{and} \quad
	\{ uv' \notin E(H) \text{ or } u'v \notin E(H) \}.
\]

\begin{proposition} \label{prop:bst-check}
A graph $H$ is a bipartite swapping target if and only if $H^\bst$ is bipartite.
\end{proposition}

\begin{proof}
To see whether $H$ is a bipartite swapping target, we only need to check that $\Hom(G \sqcup G, H) = \Hom^\bsp(G \sqcup G, H)$ for all odd cycles $G$. Indeed, if some $f \in \Hom(G \sqcup G, H)$ does not have the bipartite swapping property, then the set of violated edges with respect to $f$ contains some odd cycle $C_n$, and restriction to the cycle subgraph gives a homomorphism in $\Hom(C_n \sqcup C_n , H)$ that has all edges of $C_n$ violated.

An element in $\Hom(C_n \sqcup C_n , H)$ can be represented by closed walk of $n$ steps in $H \x H$ (i.e., through the $H$-pair-labeling). The step from $(u,v)$ to $(u',v')$ satisfies $uv \in E(H)$ and $u'v' \in E(H)$ since it is an edge of $H \x H$. Furthermore, it gives a violated edge in $C_n$ if and only if $uv' \notin E(H)$ or $u'v \notin E(H)$, and such edges form the subgraph $H^\bst \subset H\x H$. A homomorphism in $\Hom(C_n \sqcup C_n , H)$ fails to possess the bipartite swapping property if and only if there is a closed walk of $n$ steps in $H^\bst$. Checking over all odd $n$, we find that $H$ fails to possess the bipartite swapping property if and only if $H^\bst$ contains an odd cycle. The result follows.
\end{proof}

\begin{example} \label{ex:bst}
Here is a graph $H$ with $H^\bst$ drawn (indexed Cartesian-style as opposed to matrix-style). It is straightforward (although somewhat tedious) to construct the edges of $H^\bst$ using the rules given above. Note that $H^\bst$ is bipartite, so it follows that $H$ is a bipartite swapping target.
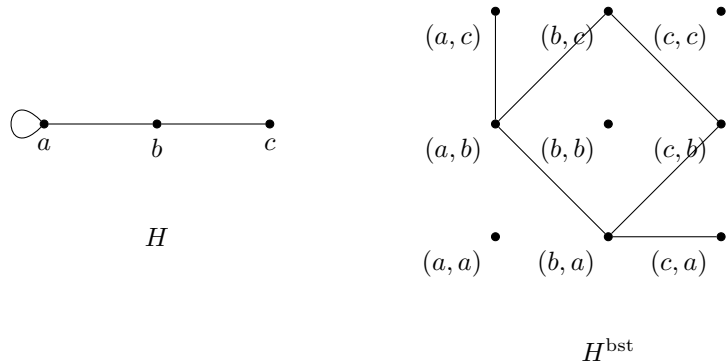
\begin{center}
\begin{tikzpicture}[sf,scale=1.5]
	\begin{scope}
		\node[P, label=below:{$a$}] (a) at (-1,0) {};
		\node[P, label=below:{$b$}] (b) at (0,0) {};
		\node[P, label=below:{$c$}] (c) at (1,0) {};
		\draw (c)--(b)--(a) edge[-,in = 135, out = 225, loop] ();
		\node at (0,-1) {$H$};
	\end{scope}
	
	\begin{scope}[xshift=4cm]
		\foreach \a/\x in {a/-1, b/0, c/1}{
			\foreach \b/\y in {a/-1, b/0, c/1}{
				\node[P, label=225:{$(\a,\b)$}] (\a\b) at (\x,\y) {};
			}
		}
		\draw (ac)--(ab)--(ba)--(cb)--(bc)--(ab);
		\draw (ba)--(ca);
		\node at (0,-2) {$H^\bst$};
	\end{scope}
\end{tikzpicture}
\end{center}
Extending this example, it turns out that if $H$ is a path with a single loop attached to either the first or the second vertex of the path, then $H^\bst$ is bipartite and thus $H$ is a bipartite swapping target. The following diagrams provide a proof-by-picture of this fact. The vertices of $H^\bst$ are drawn in the order following the example above, and they are colored black and white to show the bipartition.
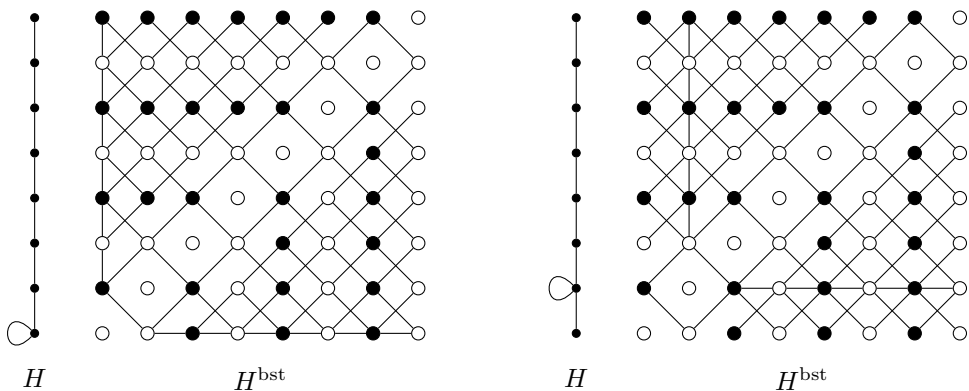
\begin{center}
\tikzstyle{B}=[draw,circle, fill=white, minimum size=5pt,inner sep=0pt]
\tikzstyle{W}=[draw,circle, fill=black, minimum size=5pt,inner sep=0pt]

\begin{tikzpicture}[scale=.6,sf]

\begin{scope}
	\begin{scope}
		\foreach \i in {0,...,7}{
			\node[P] (\i) at (0,\i) {};
		}
		\draw (7)--(6)--(5)--(4)--(3)--(2)--(1)--(0) edge[-,in = 135, out = 225, loop, distance=1cm] (0);
		\node at (0,-1) {$H$};
	\end{scope}
	
	\begin{scope}[xshift=1.5cm]
		\node at (3.5,-1) {$H^{\bst}$};
		\node[B] (7 3) at (7,3) {};
		\node[W] (4 7) at (4,7) {};
		\node[W] (1 3) at (1,3) {};
		\node[W] (6 4) at (6,4) {};
		\node[B] (3 0) at (3,0) {};
		\node[B] (5 4) at (5,4) {};
		\node[W] (0 7) at (0,7) {};
		\node[B] (5 6) at (5,6) {};
		\node[B] (2 6) at (2,6) {};
		\node[B] (1 6) at (1,6) {};
		\node[B] (5 1) at (5,1) {};
		\node[W] (3 7) at (3,7) {};
		\node[W] (2 5) at (2,5) {};
		\node[W] (0 3) at (0,3) {};
		\node[B] (7 2) at (7,2) {};
		\node[W] (4 0) at (4,0) {};
		\node[B] (1 2) at (1,2) {};
		\node[W] (6 7) at (6,7) {};
		\node[B] (3 3) at (3,3) {};
		\node[W] (2 0) at (2,0) {};
		\node[B] (7 6) at (7,6) {};
		\node[B] (4 4) at (4,4) {};
		\node[W] (6 3) at (6,3) {};
		\node[W] (1 5) at (1,5) {};
		\node[B] (3 6) at (3,6) {};
		\node[B] (2 2) at (2,2) {};
		\node[B] (7 7) at (7,7) {};
		\node[W] (5 7) at (5,7) {};
		\node[B] (5 3) at (5,3) {};
		\node[W] (4 1) at (4,1) {};
		\node[B] (1 1) at (1,1) {};
		\node[W] (2 7) at (2,7) {};
		\node[B] (3 2) at (3,2) {};
		\node[B] (0 0) at (0,0) {};
		\node[B] (6 6) at (6,6) {};
		\node[B] (5 0) at (5,0) {};
		\node[B] (7 1) at (7,1) {};
		\node[W] (4 5) at (4,5) {};
		\node[B] (0 4) at (0,4) {};
		\node[B] (5 5) at (5,5) {};
		\node[B] (1 4) at (1,4) {};
		\node[W] (6 0) at (6,0) {};
		\node[B] (7 5) at (7,5) {};
		\node[W] (2 3) at (2,3) {};
		\node[W] (2 1) at (2,1) {};
		\node[W] (4 2) at (4,2) {};
		\node[B] (1 0) at (1,0) {};
		\node[W] (6 5) at (6,5) {};
		\node[W] (3 5) at (3,5) {};
		\node[W] (0 1) at (0,1) {};
		\node[B] (7 0) at (7,0) {};
		\node[B] (4 6) at (4,6) {};
		\node[B] (5 2) at (5,2) {};
		\node[W] (6 1) at (6,1) {};
		\node[B] (3 1) at (3,1) {};
		\node[B] (0 2) at (0,2) {};
		\node[B] (7 4) at (7,4) {};
		\node[B] (0 6) at (0,6) {};
		\node[W] (6 2) at (6,2) {};
		\node[W] (4 3) at (4,3) {};
		\node[W] (1 7) at (1,7) {};
		\node[W] (0 5) at (0,5) {};
		\node[B] (3 4) at (3,4) {};
		\node[B] (2 4) at (2,4) {};
		\draw (7 3) -- (6 4);
		\draw (7 3) -- (6 2);
		\draw (4 7) -- (5 6);
		\draw (4 7) -- (3 6);
		\draw (1 3) -- (2 4);
		\draw (1 3) -- (0 2);
		\draw (1 3) -- (0 4);
		\draw (6 4) -- (7 5);
		\draw (6 4) -- (5 3);
		\draw (3 0) -- (2 0);
		\draw (3 0) -- (4 1);
		\draw (3 0) -- (2 1);
		\draw (3 0) -- (4 0);
		\draw (5 4) -- (4 5);
		\draw (5 4) -- (6 3);
		\draw (5 4) -- (6 5);
		\draw (5 4) -- (4 3);
		\draw (0 7) -- (0 6);
		\draw (0 7) -- (1 6);
		\draw (5 6) -- (4 5);
		\draw (5 6) -- (6 7);
		\draw (5 6) -- (6 5);
		\draw (2 6) -- (1 5);
		\draw (2 6) -- (3 5);
		\draw (2 6) -- (3 7);
		\draw (2 6) -- (1 7);
		\draw (1 6) -- (2 7);
		\draw (1 6) -- (2 5);
		\draw (1 6) -- (0 5);
		\draw (5 1) -- (4 2);
		\draw (5 1) -- (6 2);
		\draw (5 1) -- (6 0);
		\draw (5 1) -- (4 0);
		\draw (3 7) -- (4 6);
		\draw (2 5) -- (3 4);
		\draw (2 5) -- (3 6);
		\draw (2 5) -- (1 4);
		\draw (0 3) -- (1 2);
		\draw (0 3) -- (0 2);
		\draw (0 3) -- (1 4);
		\draw (0 3) -- (0 4);
		\draw (7 2) -- (6 3);
		\draw (7 2) -- (6 1);
		\draw (4 0) -- (3 1);
		\draw (4 0) -- (5 0);
		\draw (1 2) -- (0 1);
		\draw (1 2) -- (2 3);
		\draw (1 2) -- (2 1);
		\draw (6 7) -- (7 6);
		\draw (2 0) -- (1 0);
		\draw (2 0) -- (3 1);
		\draw (7 6) -- (6 5);
		\draw (6 3) -- (7 4);
		\draw (6 3) -- (5 2);
		\draw (1 5) -- (0 6);
		\draw (1 5) -- (2 4);
		\draw (1 5) -- (0 4);
		\draw (3 6) -- (2 7);
		\draw (3 6) -- (4 5);
		\draw (5 7) -- (4 6);
		\draw (5 3) -- (4 2);
		\draw (5 3) -- (6 2);
		\draw (4 1) -- (3 2);
		\draw (4 1) -- (5 2);
		\draw (4 1) -- (5 0);
		\draw (3 2) -- (2 3);
		\draw (3 2) -- (4 3);
		\draw (3 2) -- (2 1);
		\draw (5 0) -- (6 1);
		\draw (5 0) -- (6 0);
		\draw (7 1) -- (6 2);
		\draw (7 1) -- (6 0);
		\draw (4 5) -- (3 4);
		\draw (0 4) -- (0 5);
		\draw (1 4) -- (2 3);
		\draw (1 4) -- (0 5);
		\draw (6 0) -- (7 0);
		\draw (2 3) -- (3 4);
		\draw (2 1) -- (1 0);
		\draw (4 2) -- (3 1);
		\draw (1 0) -- (0 1);
		\draw (6 5) -- (7 4);
		\draw (3 5) -- (2 4);
		\draw (3 5) -- (4 6);
		\draw (0 1) -- (0 2);
		\draw (7 0) -- (6 1);
		\draw (5 2) -- (6 1);
		\draw (5 2) -- (4 3);
		\draw (0 6) -- (0 5);
		\draw (0 6) -- (1 7);
		\draw (4 3) -- (3 4);
	\end{scope}
\end{scope}

\begin{scope}[xshift=12cm]
	\begin{scope}
		\foreach \i in {0,...,7}{
			\node[P] (\i) at (0,\i) {};
		}
		\draw (7)--(6)--(5)--(4)--(3)--(2)--(1)  edge[-,in = 135, out = 225, loop, distance=1cm] (1)--(0);
		\node at (0,-1) {$H$};
	\end{scope}
	
	\begin{scope}[xshift=1.5cm]
		\node at (3.5,-1) {$H^{\bst}$};
		\node[B] (7 3) at (7,3) {};
		\node[W] (4 7) at (4,7) {};
		\node[W] (1 3) at (1,3) {};
		\node[W] (6 4) at (6,4) {};
		\node[B] (3 0) at (3,0) {};
		\node[B] (5 4) at (5,4) {};
		\node[W] (0 7) at (0,7) {};
		\node[B] (5 6) at (5,6) {};
		\node[B] (2 6) at (2,6) {};
		\node[B] (1 6) at (1,6) {};
		\node[B] (5 1) at (5,1) {};
		\node[W] (3 7) at (3,7) {};
		\node[W] (2 5) at (2,5) {};
		\node[W] (0 3) at (0,3) {};
		\node[B] (7 2) at (7,2) {};
		\node[W] (4 0) at (4,0) {};
		\node[B] (1 2) at (1,2) {};
		\node[W] (6 7) at (6,7) {};
		\node[B] (3 3) at (3,3) {};
		\node[W] (2 0) at (2,0) {};
		\node[B] (7 6) at (7,6) {};
		\node[B] (4 4) at (4,4) {};
		\node[W] (6 3) at (6,3) {};
		\node[W] (1 5) at (1,5) {};
		\node[B] (3 6) at (3,6) {};
		\node[B] (2 2) at (2,2) {};
		\node[B] (7 7) at (7,7) {};
		\node[W] (5 7) at (5,7) {};
		\node[B] (5 3) at (5,3) {};
		\node[W] (4 1) at (4,1) {};
		\node[B] (1 1) at (1,1) {};
		\node[W] (2 7) at (2,7) {};
		\node[B] (3 2) at (3,2) {};
		\node[B] (0 0) at (0,0) {};
		\node[B] (6 6) at (6,6) {};
		\node[B] (5 0) at (5,0) {};
		\node[B] (7 1) at (7,1) {};
		\node[W] (4 5) at (4,5) {};
		\node[B] (0 4) at (0,4) {};
		\node[B] (5 5) at (5,5) {};
		\node[B] (1 4) at (1,4) {};
		\node[W] (6 0) at (6,0) {};
		\node[B] (7 5) at (7,5) {};
		\node[W] (2 3) at (2,3) {};
		\node[W] (2 1) at (2,1) {};
		\node[W] (4 2) at (4,2) {};
		\node[B] (1 0) at (1,0) {};
		\node[W] (6 5) at (6,5) {};
		\node[W] (3 5) at (3,5) {};
		\node[W] (0 1) at (0,1) {};
		\node[B] (7 0) at (7,0) {};
		\node[B] (4 6) at (4,6) {};
		\node[B] (5 2) at (5,2) {};
		\node[W] (6 1) at (6,1) {};
		\node[B] (3 1) at (3,1) {};
		\node[B] (0 2) at (0,2) {};
		\node[B] (7 4) at (7,4) {};
		\node[B] (0 6) at (0,6) {};
		\node[W] (6 2) at (6,2) {};
		\node[W] (4 3) at (4,3) {};
		\node[W] (1 7) at (1,7) {};
		\node[W] (0 5) at (0,5) {};
		\node[B] (3 4) at (3,4) {};
		\node[B] (2 4) at (2,4) {};
		\draw (7 3) -- (6 4);
		\draw (7 3) -- (6 2);
		\draw (4 7) -- (5 6);
		\draw (4 7) -- (3 6);
		\draw (1 3) -- (1 2);
		\draw (1 3) -- (2 4);
		\draw (1 3) -- (0 2);
		\draw (1 3) -- (1 4);
		\draw (1 3) -- (0 4);
		\draw (6 4) -- (7 5);
		\draw (6 4) -- (5 3);
		\draw (3 0) -- (4 1);
		\draw (3 0) -- (2 1);
		\draw (5 4) -- (4 5);
		\draw (5 4) -- (6 3);
		\draw (5 4) -- (6 5);
		\draw (5 4) -- (4 3);
		\draw (0 7) -- (1 6);
		\draw (5 6) -- (4 5);
		\draw (5 6) -- (6 7);
		\draw (5 6) -- (6 5);
		\draw (2 6) -- (1 5);
		\draw (2 6) -- (3 5);
		\draw (2 6) -- (3 7);
		\draw (2 6) -- (1 7);
		\draw (1 6) -- (2 7);
		\draw (1 6) -- (1 5);
		\draw (1 6) -- (0 5);
		\draw (1 6) -- (1 7);
		\draw (1 6) -- (2 5);
		\draw (5 1) -- (6 1);
		\draw (5 1) -- (6 0);
		\draw (5 1) -- (6 2);
		\draw (5 1) -- (4 2);
		\draw (5 1) -- (4 1);
		\draw (5 1) -- (4 0);
		\draw (3 7) -- (4 6);
		\draw (2 5) -- (3 4);
		\draw (2 5) -- (3 6);
		\draw (2 5) -- (1 4);
		\draw (0 3) -- (1 2);
		\draw (0 3) -- (1 4);
		\draw (7 2) -- (6 3);
		\draw (7 2) -- (6 1);
		\draw (4 0) -- (3 1);
		\draw (1 2) -- (0 1);
		\draw (1 2) -- (2 3);
		\draw (1 2) -- (2 1);
		\draw (6 7) -- (7 6);
		\draw (2 0) -- (3 1);
		\draw (7 6) -- (6 5);
		\draw (6 3) -- (7 4);
		\draw (6 3) -- (5 2);
		\draw (1 5) -- (1 4);
		\draw (1 5) -- (0 6);
		\draw (1 5) -- (0 4);
		\draw (1 5) -- (2 4);
		\draw (3 6) -- (2 7);
		\draw (3 6) -- (4 5);
		\draw (5 7) -- (4 6);
		\draw (5 3) -- (4 2);
		\draw (5 3) -- (6 2);
		\draw (4 1) -- (3 2);
		\draw (4 1) -- (3 1);
		\draw (4 1) -- (5 0);
		\draw (4 1) -- (5 2);
		\draw (3 2) -- (2 3);
		\draw (3 2) -- (4 3);
		\draw (3 2) -- (2 1);
		\draw (5 0) -- (6 1);
		\draw (7 1) -- (6 1);
		\draw (7 1) -- (6 2);
		\draw (7 1) -- (6 0);
		\draw (4 5) -- (3 4);
		\draw (1 4) -- (2 3);
		\draw (1 4) -- (0 5);
		\draw (2 3) -- (3 4);
		\draw (2 1) -- (1 0);
		\draw (2 1) -- (3 1);
		\draw (4 2) -- (3 1);
		\draw (1 0) -- (0 1);
		\draw (6 5) -- (7 4);
		\draw (3 5) -- (2 4);
		\draw (3 5) -- (4 6);
		\draw (7 0) -- (6 1);
		\draw (5 2) -- (6 1);
		\draw (5 2) -- (4 3);
		\draw (0 6) -- (1 7);
		\draw (4 3) -- (3 4);
	\end{scope}
\end{scope}

\end{tikzpicture}
\end{center}
On the other hand, the following graph $H$ is not a bipartite swapping target, since $H^\bst$ is not bipartite (an odd cycle is highlighted). Any graph containing $H$ as an induced subgraph is thus also not a bipartite swapping target.
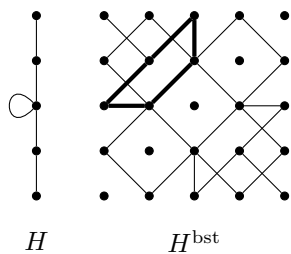
\begin{center}
\begin{tikzpicture}[scale=.6,sf]

	\begin{scope}
		\foreach \i in {0,...,4}{
			\node[P] (\i) at (0,\i) {};
		}
		\draw (4)--(3)--(2) edge[-,in = 135, out = 225, loop, distance=1cm] (2)--(1)--(0);
		\node at (0,-1) {$H$};
	\end{scope}
	
	\begin{scope}[xshift=1.5cm]
		\node at (2,-1) {$H^{\bst}$};
		\node[P] (1 3) at (1,3) {};
		\node[P] (3 0) at (3,0) {};
		\node[P] (2 1) at (2,1) {};
		\node[P] (0 3) at (0,3) {};
		\node[P] (4 0) at (4,0) {};
		\node[P] (1 2) at (1,2) {};
		\node[P] (3 3) at (3,3) {};
		\node[P] (4 4) at (4,4) {};
		\node[P] (2 2) at (2,2) {};
		\node[P] (4 1) at (4,1) {};
		\node[P] (1 1) at (1,1) {};
		\node[P] (3 2) at (3,2) {};
		\node[P] (0 0) at (0,0) {};
		\node[P] (0 4) at (0,4) {};
		\node[P] (1 4) at (1,4) {};
		\node[P] (2 3) at (2,3) {};
		\node[P] (4 2) at (4,2) {};
		\node[P] (1 0) at (1,0) {};
		\node[P] (0 1) at (0,1) {};
		\node[P] (3 1) at (3,1) {};
		\node[P] (2 4) at (2,4) {};
		\node[P] (2 0) at (2,0) {};
		\node[P] (4 3) at (4,3) {};
		\node[P] (3 4) at (3,4) {};
		\node[P] (0 2) at (0,2) {};
		\draw (1 3) -- (2 4);
		\draw (1 3) -- (0 2);
		\draw (1 3) -- (0 4);
		\draw (3 0) -- (4 1);
		\draw (3 0) -- (2 1);
		\draw (2 1) -- (1 2);
		\draw (2 1) -- (2 0);
		\draw (2 1) -- (1 0);
		\draw (2 1) -- (3 2);
		\draw (0 3) -- (1 2);
		\draw (0 3) -- (1 4);
		\draw (4 0) -- (3 1);
		\draw (1 2) -- (0 1);
		\draw (1 2) -- (2 3);
		\draw (1 2) -- (0 2);
		\draw (4 1) -- (3 2);
		\draw (3 2) -- (2 3);
		\draw (3 2) -- (4 3);
		\draw (3 2) -- (4 2);
		\draw (1 4) -- (2 3);
		\draw (2 3) -- (3 4);
		\draw (2 3) -- (2 4);
		\draw (4 2) -- (3 1);
		\draw (1 0) -- (0 1);
		\draw (3 1) -- (2 0);
		\draw (4 3) -- (3 4);
		\draw[t] (0 2) -- (1 2) -- (2 3) -- (2 4) -- (1 3) --  (0 2);
	\end{scope}

\end{tikzpicture}
\end{center}
\end{example}

\subsection{Threshold graphs} \label{sec:4c}

Bipartite swapping targets at first seem like rather elusive objects, and we are left wondering whether there are many graphs that are bipartite swapping targets. In this section we provide a simple sufficient condition for bipartite swapping targets, thereby presenting a large useful family of such graphs.

\begin{definition} \label{def:4c}
Let $H$ be a graph (not necessarily simple). An \emph{alternating 4-circuit} is a sequence $a, b, c, d \in V(H)$ (not necessarily distinct), such that $ab, cd \in E(H)$, and $bc, da \notin E(H)$.
\end{definition}

\begin{proposition} \label{prop:4c-bst}
Let $H$ be a graph (not necessarily simple). Suppose $H$ has no alternating 4-circuit, then $H$ is a bipartite swapping target.
\end{proposition}

\begin{proof}
We use Proposition \ref{prop:bst-check}. In $H^\bst$, let
\[
	W = \{ (u,v) \in V(H^\bst) = V(H) \x V(H) : uv' \notin E(H) \text{ for some } (u,v)(u',v') \in E(H^\bst) \}.
\]
We claim that every edge of $H^\bst$ has exactly one endpoint in $W$. For edge $e = (u,v)(u',v') \in E(H^\bst)$, by definition either $uv' \notin E(H)$ or $u'v \notin E(H)$, so at least one of the endpoints of $e$ is in $W$. Now suppose that both endpoints of $e$ are in $W$. Without loss of generality assume that $uv' \notin E(H)$. Since $(u', v') \in W$, we have $u'v'' \notin E(H)$ for some $(u', v')(u'', v'') \in E(H^\bst)$. Then $uv', u' v'' \notin E(H)$ and $uu', v'v'' \in E(H)$, so that $u, u', v'', v'$ is an alternating 4-circuit of $H$, contradiction. Therefore, every edge of $H^\bst$ has exactly one endpoint in $W$ and hence $H^\bst$ is bipartite.
\end{proof}

\begin{remark}
The graph $H$ in Example~\ref{ex:bst} has an alternating 4-circuit (namely $b,c,c,b$), but it is still a bipartite swapping target. Thus the converse of Proposition~\ref{prop:4c-bst} is false.
\end{remark}

Now we construct a family of graphs which have no alternating 4-circuits and are hence bipartite swapping targets.

\begin{definition}[Threshold graphs] \label{def:H}
Let $A$ be a finite (multi)set of real numbers, and $t$ be some ``threshold'' constant. Let $H_{A,t}$ denote the graph with $A$ as the vertices, and an edge between $x,y \in A$ (possibly $x=y$) if and only if $x+y \leq t$. We call such graphs \emph{threshold graphs}. When $A = \{0, 1, \dots, n\}$ and $t = n$, we write $H_n$ for $H_{A, t}$.  
\end{definition}

\begin{lemma} \label{lem:threshold-4c}
For any $A$ and $t$, the graph $H_{A, t}$ has no alternating $4$-circuit.
\end{lemma}

\begin{proof}
An alternating 4-circuit $a, b, c, d \in V(H_{A, t}) = A$ in $H_{A, t}$ must satisfy $a + b \leq t$, $c + d \leq t$, $b + c > t$, $d + a > t$. The sum of the first two inequalities give $a + b + c + d \leq 2t$ while the sum of the last two inequalities give $a + b + c + d > 2t$, which is impossible.
\end{proof}

Note that a graph homomorphism in $\Hom(G, H_{A, t})$ corresponds to assigning each vertex of $G$ some ``state'' represented by a real number in $A$, so that the sum of the states of the two endpoints of an edge never exceeds some threshold. This interpretation allows us to prove the result about generalized independent sets stated in Section \ref{summary-gen-indep}.

\begin{proof}[Proof of Theorem \ref{thm:i}]
The key observation is that $\Hom(G, H_n) \cong \c I(G, n)$ (defined in Section \ref{summary-gen-indep}). Then Theorem \ref{thm:i} is equivalent to the statement that $H_n$ is GT, which is true since $H_n$ has no alternating 4-circuit, and hence is a bipartite swapping target.
\end{proof}

The statement at the end of Section \ref{summary-gen-indep} about assignments $f : V \to A$ follows analogously by using $H = H_{A, t}$.

Next we give a complete characterization of all threshold graphs. It turns out that they are precisely the class of graphs without alternating 4-circuits.

\begin{theorem}[Characterization of threshold graphs] \label{thm:4c}
Let $H$ be a graph (allowing loops) with $n$ vertices. The following are equivalent:
\begin{enumerate}
	\item[(a)] $H$ has no alternating 4-circuit.
	\item[(b)] $H$ is isomorphic to some threshold graph $H_{A,t}$.
	\item[(c)] The vertices of $H$ can be ordered in a way so that the set of positions of the $1$'s in the adjacency matrix of $H$ form a self-conjugate Young diagram (English style).
	\item[(d)] The vertices of $H$ can be ordered as $v_1, \dots, v_{n}$ so that $N(v_1) \supseteq N(v_2) \supseteq \cdots \supseteq N(v_n)$, where $N(v)$ denotes the set of neighbors of $v$.
\end{enumerate}
\end{theorem}

\begin{remark}
The condition in (c) means that the adjacency matrix of $H$ has the property that, whenever an entry is $1$, all the entries above and/or to the left of it are all $1$'s. Self-conjugate means that matrix is symmetric, which is automatic for undirected graphs. Here is an example of a matrix satisfying (c):
	\[
		\begin{pmatrix}
			1&1&1&1&1&0\\
			1&1&1&0&0&0\\
			1&1&0&0&0&0\\
			1&0&0&0&0&0\\
			1&0&0&0&0&0\\
			0&0&0&0&0&0
		\end{pmatrix}.
	\]
Figure \ref{fig:4c-ex} shows all isomorphism classes of graphs with up to 3 vertices satisfying the conditions of Theorem \ref{thm:4c}.
\end{remark}

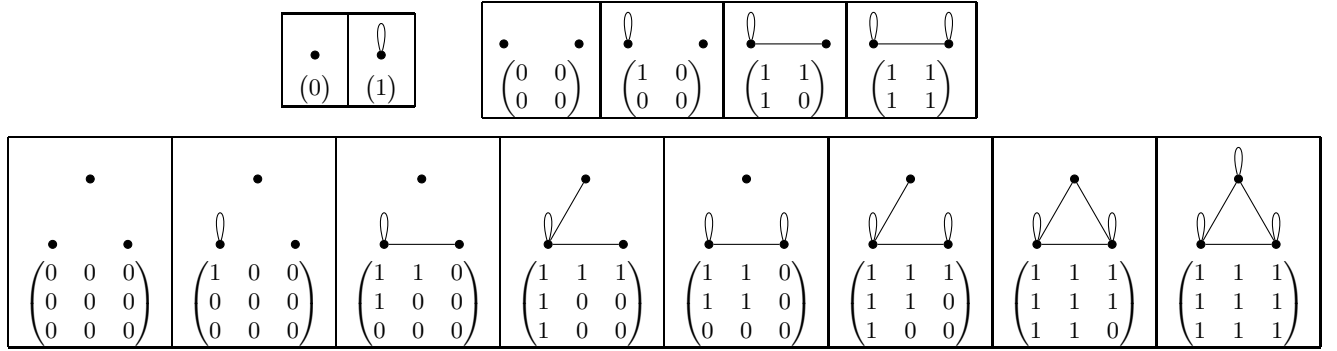
\begin{figure}[ht!]
\centering

\begin{tabular}{|c|c|}
\hline
\begin{tikzpicture}
	\node[P] (A) at (0,0) {};
\end{tikzpicture}
&
\begin{tikzpicture}
	\node[P] (A) at (0,0) {};
	\draw (A) edge[loop above] ();
\end{tikzpicture}
\\
\footnotesize $\begin{pmatrix} 0 \end{pmatrix}$ &
\footnotesize $\begin{pmatrix} 1 \end{pmatrix}$ \\
\hline
\end{tabular}
\qquad
\begin{tabular}{|c|c|c|c|}
\hline
\begin{tikzpicture}
	\node[P] (A) at (0,0) {};
	\node[P] (B) at (1,0) {};
\end{tikzpicture}
&
\begin{tikzpicture}
	\node[P] (A) at (0,0) {};
	\node[P] (B) at (1,0) {};
	\draw (A) edge[loop above] ();
\end{tikzpicture}
&
\begin{tikzpicture}
	\node[P] (A) at (0,0) {};
	\node[P] (B) at (1,0) {};
	\draw (A) edge[loop above] ();
	\draw (A) -- (B);
\end{tikzpicture}
&
\begin{tikzpicture}
	\node[P] (A) at (0,0) {};
	\node[P] (B) at (1,0) {};
	\draw (A) edge[loop above] ();
	\draw (B) edge[loop above] ();
	\draw (A) -- (B);
\end{tikzpicture}
\\
\footnotesize $\begin{pmatrix} 0&0\\0&0 \end{pmatrix}$ &
\footnotesize $\begin{pmatrix} 1&0\\0&0 \end{pmatrix}$ &
\footnotesize $\begin{pmatrix} 1&1\\1&0 \end{pmatrix}$ &
\footnotesize $\begin{pmatrix} 1&1\\1&1 \end{pmatrix}$ \\ \hline
\end{tabular}

\medskip

\begin{tabular}{|c|c|c|c|c|c|c|c|}
\hline
\begin{tikzpicture}
	\node[P] (A) at (0,0) {};
	\node[P] (B) at (1,0) {};
	\node[P] (C) at (0.5, 0.87) {};
\end{tikzpicture}
&
\begin{tikzpicture}
	\node[P] (A) at (0,0) {};
	\node[P] (B) at (1,0) {};
	\node[P] (C) at (0.5, 0.87) {};
	\draw (A) edge[loop above] ();
\end{tikzpicture}
&
\begin{tikzpicture}
	\node[P] (A) at (0,0) {};
	\node[P] (B) at (1,0) {};
	\node[P] (C) at (0.5, 0.87) {};
	\draw (A) edge[loop above] ();
	\draw (A) -- (B);
\end{tikzpicture}
&
\begin{tikzpicture}
	\node[P] (A) at (0,0) {};
	\node[P] (B) at (1,0) {};
	\node[P] (C) at (0.5, 0.87) {};
	\draw (A) edge[loop above] ();
	\draw (A) -- (B);
	\draw (A) -- (C);	
\end{tikzpicture}
&
\begin{tikzpicture}
	\node[P] (A) at (0,0) {};
	\node[P] (B) at (1,0) {};
	\node[P] (C) at (0.5, 0.87) {};
	\draw (A) edge[loop above] ();
	\draw (B) edge[loop above] ();
	\draw (A) -- (B);
\end{tikzpicture}
&
\begin{tikzpicture}
	\node[P] (A) at (0,0) {};
	\node[P] (B) at (1,0) {};
	\node[P] (C) at (0.5, 0.87) {};
	\draw (A) edge[loop above] ();
	\draw (B) edge[loop above] ();
	\draw (A) -- (B); \draw (A)--(C);
\end{tikzpicture}
&
\begin{tikzpicture}
	\node[P] (A) at (0,0) {};
	\node[P] (B) at (1,0) {};
	\node[P] (C) at (0.5, 0.87) {};
	\draw (A) edge[loop above] ();
	\draw (B) edge[loop above] ();
	\draw (A) -- (B); \draw (A)--(C); \draw (B)--(C);
\end{tikzpicture}
&
\begin{tikzpicture}
	\node[P] (A) at (0,0) {};
	\node[P] (B) at (1,0) {};
	\node[P] (C) at (0.5, 0.87) {};
	\draw (A) edge[loop above] ();
	\draw (B) edge[loop above] ();
	\draw (C) edge[loop above] ();
	\draw (A) -- (B); \draw (A)--(C); \draw (B)--(C);
\end{tikzpicture}
\\
\footnotesize $\begin{pmatrix} 0&0&0\\0&0&0\\0&0&0 \end{pmatrix}$ &
\footnotesize $\begin{pmatrix} 1&0&0\\0&0&0\\0&0&0 \end{pmatrix}$ &
\footnotesize $\begin{pmatrix} 1&1&0\\1&0&0\\0&0&0 \end{pmatrix}$ &
\footnotesize $\begin{pmatrix} 1&1&1\\1&0&0\\1&0&0 \end{pmatrix}$ &
\footnotesize $\begin{pmatrix} 1&1&0\\1&1&0\\0&0&0 \end{pmatrix}$ &
\footnotesize $\begin{pmatrix} 1&1&1\\1&1&0\\1&0&0 \end{pmatrix}$ &
\footnotesize $\begin{pmatrix} 1&1&1\\1&1&1\\1&1&0 \end{pmatrix}$ &
\footnotesize $\begin{pmatrix} 1&1&1\\1&1&1\\1&1&1 \end{pmatrix}$ \\
\hline
\end{tabular}
\caption{All graphs with up to 3 vertices satisfying Theorem~\ref{thm:4c} and their adjacency matrices.\label{fig:4c-ex}} 
\end{figure}

\begin{proof}[Proof of Theorem~\ref{thm:4c}] We will show that (a)$\Leftarrow$(b)$\Leftarrow$(c)$\Leftarrow$(d)$\Leftarrow$(a). The implication (b)$\Rightarrow$(a) has already been established in Lemma~\ref{lem:threshold-4c}. 


(c)$\Rightarrow$(b): 
Start with an adjacency matrix satisfying (c). Let $r_i$ denote the number of $1$'s in the $i$-th row. Let $a_i = i - r_i$, and let $A$ denote the multiset $\set{a_1, a_2, \dots, a_n}$. We claim that $H$ is isomorphic to $H_{A,0}$, where $a_i \in A$ corresponds to the vertex represented by the $i$-th row of the matrix. Indeed, if the $(i,j)$ entry in the matrix is $1$, then $i \leq r_j$ and $j \leq r_i$, so that $a_i + a_j = i - r_i + j - r_j \leq 0$. Otherwise, the $(i, j)$ entry is $0$, so $i > r_j$ and $j > r_i$, and hence $a_i + a_j = i - r_i + j - r_j > 0$.

(d)$\Rightarrow$(c): Suppose that (d) holds. We claim that the adjacency matrix of $H$ with respect to the vertex ordering $v_1, \dots, v_n$ satisfies (c). It suffices to show that if the entry $(i,j)$ of adjacency matrix is $1$ (denoting $v_iv_j \in E(H)$), then every entry directly above or directly to the left of it is $1$. Due to symmetry, we only need to consider the entries above $(i,j)$. For $k < j$, we have $N(v_k) \supseteq N(v_j) \ni v_i$, so $v_iv_k \in E(H)$ and hence the entry at $(i,k)$ is $1$. This shows that (c) is satisfied.

(a)$\Rightarrow$(d): Suppose that $H$ has no alternating 4-circuit. Order the vertices by decreasing degree, so that $\abs{N(v_1)} \geq \abs{N(v_2)} \geq \cdots \geq \abs{N(v_n)}$. We claim that (d) is satisfied for this ordering. Suppose not, so that $N(v_i) \nsupseteq N(v_j)$ for some $i < j$. Since $\abs{N(v_i)} \geq \abs{N(v_j)}$, we have $N(v_i) \nsubseteq N(v_j)$ as well. Let $x \in N(v_i) \setminus N(v_j)$ and $y \in N(v_j) \setminus N(v_i)$. Then $v_i, x, v_j, y$ is an alternating 4-circuit. Contradiction. Therefore, (d) is satisfied.
\end{proof}

We conclude this section with an enumerative result about threshold graphs, thereby showing the abundunce of bipartite swapping targets.

\begin{proposition} \label{prop:4c-enum}
There are exactly $\binom{n}{k}$ isomorphism classes of threshold graphs $n$ vertices and have exactly $k$ loops, and there are exactly $2^n$ isomorphism classes of threshold graphs on $n$ vertices.
\end{proposition}

\begin{proof}
Using characterization (c) of Theorem \ref{thm:4c}, we see that graphs with $k$ loops correspond bijectively to paths on the Euclidean lattice $(0,0)$ to $(k, n-k)$ using steps $(1,0)$ and $(0,1)$ (that is, consider the boundary between the 0's and the 1's up to the diagonal of the matrix) and there are exactly $\binom{n}{k}$ such walks. The second statement follows from summing over all $k$.
\end{proof}

\begin{remark}
The $\binom{n}{k}$ classes can be constructed by arranging $k$ looped vertices and $n-k$ non-looped vertices in a row, and then connecting every looped vertex to all the vertices on its right.
\end{remark}


\section{Counting graph colorings} \label{sec:coloring}

When $H = K_q$, the set $\Hom(G, K_q)$ is in bijective correspondence with proper vertex colorings of $G$ with $q$ colors corresponding to the vertices of $K_q$. The number of proper $q$-colorings of $G$ is equal to $P(G, q) = \hom(G, K_q)$, the chromatic polynomial of $G$. As discussed in Section \ref{sec:summary-coloring}, we suspect that $K_q$ is GT, so that $P(G, q) \leq P(K_{d,d}, q)^{N/(2d)}$. Unfortunately, when $q \geq 3$, $K_q$ is not a bipartite swapping target, since it contains an induced triangle, which is not a bipartite swapping target by Example \ref{ex:odd-cycle-not-target}. Nevertheless, we still suspect that $K_q$ is strongly GT.

\begin{conjecture} \label{conj:K_q-strongly-GT}
$K_q$ is strongly GT.
\end{conjecture}

Note that Conjecture \ref{conj:K_q-strongly-GT} implies Conjecture \ref{conj:coloring}. Since $K_q$ is not a bipartite swapping target, we cannot directly apply the bipartite swapping trick. However, it turns out that we can still use the bipartite swapping trick to prove an asymptotic version of Conjecture~\ref{conj:K_q-strongly-GT}. Here is the main result of this section.

\begin{proposition} \label{prop:color-asymp-bipartite}
Let $G$ be a graph with $N$ vertices. Then $P(G \sqcup G, q) \leq P(G \x K_2, q)$ for $q \geq (2N)^{2N + 2}$.
\end{proposition}

Before we prove Proposition~\ref{prop:color-asymp-bipartite}, let us deduce Theorem~\ref{thm:asym-coloring} from the Proposition.

\begin{proof}[Proof of Theorem~\ref{thm:asym-coloring}]
From Proposition \ref{prop:color-asymp-bipartite}, we have $P(G, q)^2 = P(G \sqcup G, q) \leq P(G \x K_2, q)$ for sufficiently large $q$. Theorem \ref{thm:GT} implies that $P(G \x K_2, q) = \hom(G \x K_2, K_q) \leq \hom(K_{d,d}, K_q)^{N/d} = P(K_{d,d}, K_q)^{N/d}$ for all $q$. Theorem~\ref{thm:asym-coloring} then follows from combining the two inequalities.
\end{proof}

\begin{remark}
After the initial draft of this paper was written, F.~Lazebnik observed (personal communication to the author via D.~Galvin) that $P(G, q) \leq P(K_{d,d},q)^{N/(2d)}$ whenever $N/(2d)$ is an integer and $q > 2 \binom{nd/2}{4}$, thereby improving the lower bound on $q$ in Theorem~\ref{thm:asym-coloring} at least in the case when $N$ is divisible by $2d$. This proof uses a completely different method from this paper, and is inspired by Lazebnik's \cite{Lazebnik91} use of the Whitney broken circuit characterization of the chromatic polynomial.
\end{remark}

Let $\Hom^\surj(G, H)$ denote the subset of $\Hom(G, H)$ containing homomorphisms whose maps of vertices $V(G) \to V(H)$ is surjective. Also let $\hom^\surj(G, H) = \abs{\Hom^\surj(G, H)}$. We know that
\[
	P(G, q) = \sum_{i=0}^{\abs{V(G)}} \hom^\surj(G, K_i) \binom{q}{i}.
\]
Indeed, if exactly $i$ colors are used in the coloring, then there are $\binom{q}{i}$ ways to choose the $i$ colors used, and $\hom^\surj(G, K_i)$ ways to color $G$ using all $i$ colors. Now
\begin{align}
	P(G \sqcup G, q) &= \sum_{i=0}^{2N} \hom^\surj(G \sqcup G, K_i) \binom{q}{i},  \label{eq:chrompoly1}\\
	\text{and}  \qquad P(G \x K_2, q) &= \sum_{i=0}^{2N} \hom^\surj(G \x K_2, K_i) \binom{q}{i}. \label{eq:chrompoly2}
\end{align}
From playing with small examples, it seems that the $P(G\sqcup G, q) \leq P(G \x K_2, q)$ holds even when \eqref{eq:chrompoly1} and \eqref{eq:chrompoly2} are compared term-by-term. We state this as a conjecture. Observe that Conjecture~\ref{conj:K_q-strongly-GT} follows from this stronger conjecture.

\begin{conjecture} \label{conj:surj-coloring}
  If $G$ is a simple graph, then for all positive integers $i$,
\[
\hom^\surj(G \sqcup G, K_i) \leq \hom^\surj(G \x K_2, K_i).
\]
\end{conjecture}

Although we are unable to prove Conjecture~\ref{conj:surj-coloring}, we will prove the inequality for the most significant terms of \eqref{eq:chrompoly1} and \eqref{eq:chrompoly2}. Note that $\binom{q}{i}$ is a polynomial in $q$ of degree $i$. If $G$ is bipartite, then $G \x K_2 \cong G \sqcup G$, so the two polynomials \eqref{eq:chrompoly1} and \eqref{eq:chrompoly2} are equal. So we shall assume that $G$ is non-bipartite. Our strategy is to compare the coefficients of $\binom{q}{i}$ in \eqref{eq:chrompoly1} and \eqref{eq:chrompoly2} starting from the highest $i$, and show that on the first instance when the two coefficients differ, the coefficient in \eqref{eq:chrompoly2} is greater. This would imply that $P(G, q)^2 < P(G \x K_2, q)$ for large $q$. Specifically, we claim the following.

\begin{lemma} \label{lem:coef-compare}
Suppose that $G$ has $N$ vertices and odd girth $t$, then
\begin{align*}
	\hom^\surj(G \sqcup G, K_i) &= \hom^\surj(G \x K_2, K_i) \quad \text{for $i \geq 2N - t + 2$}, \\
	\text{and} \qquad \hom^\surj(G \sqcup G, K_i) &< \hom^\surj(G \x K_2, K_i) \quad \text{for $i = 2N - t + 1$}.
\end{align*}
\end{lemma}

The proof of Lemma~\ref{lem:coef-compare} requires several more lemmas.

\begin{lemma} \label{lem:cycle-violate1}
If $f \in \Hom^\surj(G \sqcup G, K_i)$, and the set of violated edges of $G$ with respect to $f$ contains a cycle of length $\ell$, then $i \leq 2N - \ell + 1$, where $N = \abs{V(G)}$. Furthermore, if $\ell$ is odd, then $i \leq 2N - \ell$.
\end{lemma}

\begin{proof}
For each color $c \in V(K_i)$, let $\abs{f^{-1}(c)}$ denote the number of vertices of $G \sqcup G$ colored using $c$. Then $\sum_{c \in V(K_i)} \abs{f^{-1}(c)} = 2N$, so that
\begin{equation} \label{eq:repeated-colors}
	\sum_{c \in V(K_i)} (\abs{f^{-1}(c)}-1) = 2N - i.
\end{equation}
Let $v^1, \dots, v^{\ell} \in V(G)$ be the cycle of violated edges. As we color each pair of vertices $(v^j_0, v^j_1)$ with a pair of colors in the order $j=1,2,\dots,\ell$, the condition that the edge $v^jv^{j+1}$, for $1 \leq i \leq \ell - 1$, is violated implies that in order to color the pair $(v^{j+1}_0, v^{j+1}_1)$ after having colored $(v^j_0, v^j_1)$, some previously used color must be repeated at least one more time, thereby contributing at least one to the sum on the left-hand side of \eqref{eq:repeated-colors}. Since this is the case for each $1 \leq i \leq \ell - 1$, it follows that the left-hand side of \eqref{eq:repeated-colors} is at least $\ell - 1$. Thus $2N - i \geq \ell - 1$, thereby showing the first statement in the lemma.

If $\ell$ is odd, then the final edge in the cycle $v^{\ell}v^1$ must also contribute one more repeated color, thereby showing that $2N - i \geq \ell$. (This is not the case for $\ell$ even because we can use the same color for $v^1_0, v^2_1, v^3_0, v^4_1, \dots, v^\ell_1$, and different and distinct colors for all other vertices of $G \sqcup G$).
\end{proof}

There is a parallel lemma for $G \x K_2$, whose proof we omit since it is completely analogous to the first part of Lemma \ref{lem:cycle-violate1}.

\begin{lemma} \label{lem:cycle-violate2}
If $f \in \Hom^\surj(G \x K_2, K_i)$, and the set of violated edges of $G$ with respect to $f$ contains a cycle of length $\ell$, then $i \leq 2N - \ell + 1$.
\end{lemma}

\begin{lemma} \label{lem:surj-bsp1}
Suppose that $G$ has $N$ vertices and odd girth $t$, then every element of $\Hom^\surj(G \sqcup G, K_i)$ for $i \geq 2N - t + 1$ has the bipartite swapping property.
\end{lemma}

\begin{proof}
Suppose that some $f \in \Hom^\surj(G \sqcup G, K_i)$ fails to have the bipartite swapping property, then the set of violated edges contains an odd $\ell$-cycle, and $\ell \geq t$ since $t$ is the odd girth of $G$. Then Lemma \ref{lem:cycle-violate1} implies that $i \leq 2N - \ell \leq 2N - t$, which contradicts $i \geq 2N - t + 1$.
\end{proof}

\begin{lemma} \label{lem:surj-bsp2}
Suppose that $G$ has $N$ vertices and odd girth $t$, then every element of $\Hom^\surj(G \x K_2, K_i)$ for $i \geq 2N - t + 2$ has the bipartite swapping property. Furthermore, some element of $\Hom^\surj(G \x K_2, K_i)$ for $i = 2N - t + 1$ does not have the bipartite swapping property.
\end{lemma}

\begin{proof}
The first part is analogous to Lemma \ref{lem:surj-bsp1}. For the second part, suppose that $v^1, \dots, v^t$ is a $t$-cycle in $G$. Consider the coloring of $G \x K_2$ which colors $v_0^1, v_0^2, \dots, v_0^t$ all with the same color, and all other vertices of $G \x K_2$ with different and distinct colors. Then the odd $t$-cycle is violated, and exactly $2N-t+1$ colors are used.
\end{proof}

Now we are ready to apply the bipartite swapping trick.

\begin{lemma} \label{lem:hom-surj-inject}
Suppose that $G$ has $N$ vertices and odd girth $t$, then for $i\geq 2N-t+1$, the bipartite swapping trick gives an injective map
\[
	\phi : \Hom^\surj(G \sqcup G, K_i) \longto \Hom^\surj(G \x K_2, K_i).
\]
This map is a bijection when $i \geq 2N - t + 2$, but fails to be a surjection when $i = 2N - t + 1$.
\end{lemma}

\begin{proof}
From Lemma \ref{lem:surj-bsp1}, we see that when $i\geq 2N-t+1$, $\Hom^\surj(G \sqcup G, K_i)$ is a subset of $\Hom^\bsp(G \sqcup G, K_i)$, so that we can bijectively map it to a subset of $\Hom^\surj(G \x K_2, K_i)$. Note that the bipartite swapping trick preserves the surjectivity of the map of the vertices, so the image of $\Hom^\surj(G \sqcup G, K_i)$ lies in $\Hom^\surj(G \x K_2, K_i)$, and hence $\phi$ is an injection.

When $i \geq 2N-t +2$, from Lemma \ref{lem:surj-bsp2} we know that $\Hom^\surj(G \x K_2, K_i) \subseteq \Hom^\bsp(G \x K_2, K_i)$, so that we can apply the bipartite swapping trick to $\Hom^\surj(G \x K_2, K_i)$ to obtain the inverse of $\phi$. 

When $i = 2N - t + 1$, from Lemma \ref{lem:surj-bsp2} we know that some element of $\Hom^\surj(G \x K_2, K_i)$ does not have the bipartite swapping property. Therefore, $\phi$ is not surjective.
\end{proof}

Lemma \ref{lem:coef-compare} follows immediately from Lemma \ref{lem:hom-surj-inject}.

\begin{proof}[Proof of Proposition \ref{prop:color-asymp-bipartite}]
If $G$ is bipartite, the $G \x K_2 \cong G \sqcup G$, so $P(G, q)^2 = P(G \x K_2, q)$. Otherwise, let $t$ be the odd girth of $G$. Then using Lemma \ref{lem:coef-compare} and equations \eqref{eq:chrompoly1} and  \eqref{eq:chrompoly2}, we find that for $q \geq (2N)^{2N+2}$,
\begin{align*}
	P(G \x K_2, q) - P(G\sqcup G, q)
&	=  \sum_{i=0}^{2N} (\hom^\surj(G \x K_2, K_i) - \hom^\surj(G \sqcup G, K_i)) \binom{q}{i}
\\&	\geq \binom{q}{2N-t + 1} +  \sum_{i=0}^{2N-t} (\hom^\surj(G \x K_2, K_i) - \hom^\surj(G \sqcup G, K_i)) \binom{q}{i}
\\&	\geq \binom{q}{2N-t + 1} -  \sum_{i=0}^{2N-t} \hom^\surj(G \sqcup G, K_i) \binom{q}{i}
\\&	\geq \binom{q}{2N-t + 1} -  \sum_{i=0}^{2N-t} i^{2N} \binom{q}{i}
\\&	\geq \binom{q}{2N-t + 1} - (2N-t+1)(2N-t)^{2N} \binom{q}{2N-t}
\\&	\geq \( \frac{q - 2N + t}{2N - t + 1} - (2N-t+1)(2N-t)^{2N} \) \binom{q}{2N-t}
\end{align*}
which is nonnegative as long as
\[
	q \geq (2N - t + 1)^2 (2N - t)^{2N} + 2N - t.
\]
Note that $t \geq 3$, so $q \geq (2N)^{2N+2}$ suffices.
\end{proof}


\section{Stable set polytope} \label{sec:polytope}

Let $G$ be a graph. For any $S \subseteq V(G)$, let $\b 1_S \in \RR^V$ be the characteristic vector of $S$, i.e., the component of $\b 1_S$ corresponding to $v \in V(G)$ is 1 if $v \in S$ and 0 otherwise. The \emph{stable set polytope} $\STAB(G)$ of $G$ is defined to be the convex hull of the characteristic vectors of all independent sets of $G$, i.e.,
\[
	\STAB(G) = \conv\set{\b 1_I : I \in \c I(G)}.
\]
For instance, $\STAB(K_3)$ is the tetrahedron with vertices $(0,0,0), (1,0,0), (0,1,0), (0,0,1)$. 
For every $I \in \c I(G)$, $\bx = \b 1_I$ satisfies
\begin{align}
	0 \leq x_v &\leq 1 \quad \forall v \in V(G)  \label{eq:estab1}, \text{ and }\\
	x_u + x_v &\leq 1 \quad \forall uv \in E(G) \label{eq:estab2}. 
\end{align}
It follows that every point in $\STAB(G)$ also satisfies \eqref{eq:estab1} and \eqref{eq:estab2}, and hence $\STAB(G)$ is contained in the polytope
\[
	\ESTAB(G) = \set{ \b x \in \RR^V : \b x = (x_v) \text{ satisfies \eqref{eq:estab1} and \eqref{eq:estab2}}}.
\]
Although we always have $\STAB(G) \subseteq \ESTAB(G)$, the containment may be strict. For instance, $(\tfrac12,\tfrac12,\tfrac12)$ lies in $\ESTAB(K_3)$ but not $\STAB(K_3)$. It is well-known that the two polytopes are equal if and only if $G$ is bipartite.

\begin{theorem}\cite[Thm.~19.7]{SchA} \label{thm:stab=estab}
For any graph $G$, $\STAB(G) \subseteq \ESTAB(G)$, with equality if and only if $G$ is bipartite.
\end{theorem}

Let $\vol(\c P)$ denote the volume of a polytope $\c P$. Recall from Section \ref{sec:summary-vol} the notation $i_V(G) = \vol(\STAB(G))$. So we have
\[
	i_V(G) = \vol(\STAB(G)) \leq \vol(\ESTAB(G)),
\]
with equality if $G$ is bipartite. Thus the inequality in Theorem \ref{thm:iV} follows from the following stronger statement, which is what we will prove.

\begin{proposition} \label{prop:estab-vol}
For any $N$-vertex, $d$-regular graph $G$,
\[
	\vol(\ESTAB(G)) \leq  \vol(\ESTAB(K_{d,d}))^{N/(2d)}.
\]
\end{proposition}

For a polytope $\c P$, let $n \c P$ denote the image of $\c P$ after a dilation at the origin by a factor $n$. So, 
\[
	n\ESTAB(G) = \set{ \bx \in \RR^V : 0 \leq x_v \leq n \; \forall v \in V(G), \; x_u + x_v \leq n \; \forall uv \in E(G)}.
\]
Since
\[
	\c I(G, n) = \set{ x : V(G) \to \set{0, 1, \dots, n} :  x(u) + x(v) \leq n \; \forall uv \in E(G)},
\]
lattice points in $n\ESTAB(G)$ correspond bijectively with $\c I(G, n)$. Hence
\[
	i(G, n) = \abs{\c I(G, n)} = \abs{(n \ESTAB(G)) \cap \ZZ^{V(G)}}.
\]
Regarded as a function in $n$, $i(G, n)$ is known as the Ehrhart polynomial of the polytope $\ESTAB(G)$. It is related to the volume of $\ESTAB(G)$ by
\begin{equation} \label{eq:ehrhart-vol}
	\vol(\ESTAB(G)) = \lim_{n \to \infty} i(G, n) n^{-\abs{V(G)}}.
\end{equation}


\begin{proof}[Proof of Proposition \ref{prop:estab-vol}]
From Theorem \ref{thm:i} we have
\[
	i(G, n) n^{-N} \leq \( i(K_{d,d}, n) n^{-2d} \)^{N/(2d)}.
\]
Letting $n \to \infty$ and using \eqref{eq:ehrhart-vol} gives
\[
	\vol(\ESTAB(G)) \leq \vol(\ESTAB(K_{d,d}))^{N/(2d)}. \qedhere
\]
\end{proof}

\begin{lemma} \label{lem:iV-K}
Let $a$ and $b$ be positive integers. Then
	$\displaystyle i_V(K_{a,b}) = \binom{a+b}{a}^{-1}$.
\end{lemma}

\begin{proof}
Label the coordinates of $\RR^{V(K_{a,b})}$ by $(x_1, \dotsc, x_a, y_1, \dotsc, y_b)$, so that $\STAB(K_{a,b})$ is the convex hull of points of the form $(x_1, \dots, x_a, 0, \dots, 0)$ or $(0, \dots, 0, y_1, \dots, y_b)$, where $x_i, y_j \in \{0, 1\}$.

For each pair of permutations $(\pi, \sigma) \in S_a \x S_b$, consider the subset $T_{\pi, \sigma}$ of $\STAB(K_{a,b})$ lying in the region defined by
\[
	x_{\pi(1)} \leq x_{\pi(2)} \leq \cdots \leq x_{\pi(a)} \quad \text{and} \quad
	y_{\sigma(1)} \leq y_{\sigma(2)} \leq \cdots \leq y_{\sigma(a)}.
\]
Note that $\{ T_{\pi, \sigma} : (\pi, \sigma) \in  S_a \x S_b\}$ gives a dissection of $\STAB(K_{a,b})$. Indeed, excluding the measure-zero set of points with some two coordinates equal, the first $a$ coordinates and the last $b$ coordinates of every point can be ordered in a unique way, thereby obtaining a unique $\pi$ and $\sigma$.

By symmetry, all $T_{\pi, \sigma}$ are congruent, so we can consider the one where both $\pi$ and $\sigma$ are  identity permutations. We see that $T_{\pi, \sigma}$ is the simplex with one vertex at the origin, and the other vertices the rows of the matrix
\[
	\begin{pmatrix}
		U_a & 0 \\
		0 & U_b
	\end{pmatrix}
\]
where $U_n$ is the $n\x n$ upper-triangular matrix with 1's everywhere on or above the diagonal. The determinant of this matrix is $1$, so $\vol(T_{\pi, \sigma}) = 1/(a+b)!$, and this is true for all $(\pi, \sigma)$ due to symmetry. Since $\STAB(K_{a,b})$ is triangulated into $a!b!$ such simplices, we have $\vol(\STAB(K_{a,b})) = a!b!/(a+b)!$, as claimed.
\end{proof}

\begin{proof}[Proof of Theorem \ref{thm:iV}]
We have
\begin{align*}
	i_V(G) &= \vol(\STAB(G)) \leq \vol(\ESTAB(G))  \\
			&\leq \vol(\ESTAB(K_{d,d}))^{N/(2d)} = \vol(\STAB(K_{d,d}))^{N/(2d)} = i_V(K_{d,d}) 
			= \binom{2d}{d}^{-N/(2d)}. \hfill\qedhere
\end{align*}
\end{proof}

\begin{remark}
In the spirit of \cite[Thm.~4.3]{Kahn} and \cite[Prop.~1.10]{GT}, our proof can be modified to prove the following extension.
\end{remark}

\begin{proposition}
For any $(a,b)$-biregular, $N$-vertex, bipartite $G$, we have
\[
	i_V(G) \leq i_V\(K_{a,b}\)^{N/(a+b)} = \binom{a+b}{b}^{-N/(a+b)}.
\]
\end{proposition}



\section{Weighted generalizations} \label{sec:weighted}

In this section we discuss weighted generalizations of our results on graph homomorphisms. In applications in statistical mechanics and communication networks, these weights can be used to represent probabilities.

Assign to each vertex of $H$ a nonnegative real number \emph{weight} $\lambda_w$ (also known as the \emph{activity} or \emph{fugacity}). For any $f \in \Hom(G, H)$, the weight of $f$ is defined to be $w(f) = \prod_{v \in G} \lambda_{f(v)}$. Given a vector of weights $\Lambda =(\lambda_w : w \in V(H))$, let
\[
	\hom^\Lambda(G, H) = \sum_{f \in \Hom(G, H)} w(f).
\]
See \cite{BrWi} for the statistical mechanical motivation of this construction. When $\lambda_w = 1$ for all $w \in V(H)$, we have $\hom^\Lambda(G, H) = \hom(G, H)$. So the following result is a weighted generalization of Theorem \ref{thm:GT}.

\begin{theorem}[Galvin-Tetali \cite{GT}] \label{thm:wGT}
For any $N$-vertex, $d$-regular bipartite graph $G$, any $H$ (possibly with loops), and any vector of nonnegative weights $\Lambda$ on $V(H)$, we have
\begin{equation} \label{eq:wGT}
	\hom^\Lambda(G, H) \leq \(\hom^\Lambda(K_{d,d}, H)\)^{N/(2d)}.
\end{equation}
\end{theorem}

We would like to know when Theorem~\ref{thm:wGT} can be extended to non-bipartite graphs as well.

\begin{definition} \label{def:wGT}
A graph $H$ (not necessarily simple) is \emph{wGT} if 
\begin{equation} \label{eq:wGT-def}
	\hom^\Lambda(G, H) \leq \hom^\Lambda(K_{d,d}, H)^{N/(2d)}
\end{equation}
holds for every $N$-vertex, $d$-regular graph $G$, and any vector of nonnegative weights $\Lambda$ on $V(H)$.
\end{definition}

\begin{definition} \label{def:swGT}
A graph $H$ (not necessarily simple) is \emph{strongly wGT} if
\begin{equation} \label{eq:swGT}
	\hom^\Lambda(G \sqcup G, H) \leq \hom^\Lambda(G \x K_2, H)
\end{equation}
for every graph $G$ (not necessarily regular), and any vector of nonnegative weights $\Lambda$ on $V(H)$.
\end{definition}

By setting unit weights, we see that wGT implies GT, and strongly wGT implies strongly GT.

\begin{lemma} \label{lem:swGT-wGT}
If $H$ is strongly wGT, then it is wGT.
\end{lemma}

The proof of the lemma is essentially the same as that of Lemma \ref{lem:sGT-GT}, so we omit it.

\begin{proposition} \label{prop:bst-swGT}
If $H$ is a bipartite swapping target, then $H$ is strongly wGT, and hence wGT.
\end{proposition}

\begin{proof}
Since swapping preserves weights, we know from Proposition~\ref{prop:injection} that there is a weight-preserving injection from $\Hom(G \sqcup G, H)$ to $\Hom(G \x K_2, H)$. This implies that $\hom^\Lambda(G \sqcup G, H) \leq \hom^\Lambda(G \x K_2, H)$, and hence $H$ is strongly wGT.
\end{proof}

We can now modify our chain of implication given in Section \ref{sec:summary-GT} as follows:
\begin{align*}
	&H \text{ is a threshold graph $H_{A,t}$ (Definition~\ref{def:H}, Theorem~\ref{thm:4c})}
\\	\overset{\text{Prop.~\ref{prop:4c-bst}}}{\Longrightarrow} \quad & H \text{ has the bipartite swapping target (Definition~\ref{def:target})}
\\	\overset{\text{Prop.~\ref{prop:bst-swGT}}}{\Longrightarrow} \quad & H \text{ is strongly wGT (Definition~\ref{def:swGT})}
\\	\overset{\text{Lem.~\ref{lem:swGT-wGT}}}{\Longrightarrow} \quad & H \text{ is wGT (Definition~\ref{def:wGT})}	
\end{align*}

Recall from Section~\ref{sec:4c} that the graph $H_{A,t}$ has no alternating 4-circuit.

\begin{corollary} \label{cor:H_n-wGT}
$H_{A,t}$ is wGT.
\end{corollary}

The fact that $H_1$ is wGT was proven in \cite{Zhao:indep}, in which Theorem \ref{thm:kahn-zhao} was proven in the following weighted form. 

\begin{theorem} \emph{\cite{Zhao:indep}} \label{thm:zhao}
For any $N$-vertex, $d$-regular graph $G$, and any $\lambda \geq 0$,
\[
	I(\lambda, G) \leq I(\lambda, K_{d,d})^{N/(2d)} = (2 (1 + \lambda)^d  - 1)^{N/(2d)},
\]
where $I(\lambda, G)$ is the stable set polynomial of $G$, given by
\[
	I(\lambda, G) = \sum_{I \in \c I(G)} \lambda^{\abs{I}}.
\]
\end{theorem}

Note that $I(\lambda, G) = \hom^{(1, \lambda)}(G, H_1)$ and $\hom^{(\lambda_1, \lambda_2)}(G, H_1) = \lambda_1^N \hom^{(1, \lambda_2/\lambda_1)}(G, H_1) = \lambda_1^N I(\lambda_2/\lambda_1, G)$. Hence Corollary \ref{cor:H_n-wGT} is a generalization of Theorem \ref{thm:zhao}.

Here is an interpretation of the wGT property applied to the graph $H_{A,t}$.

\begin{theorem} \label{thm:states-weight}
Let $S$ be a finite set of ``states,'' with attributes $\alpha : S \to \RR$ and $\lambda : S \to \RR_{\geq 0}$. Let $t$ be a real constant. For any graph $G$, let
\[
	\sigma_S(G) = \sum_{f} \prod_{v \in G} \lambda_{f(v)}
\]
where the sum is taken over all $f: V(G) \to S$ satisfying: $\alpha(f(u)) + \alpha(f(v)) \leq t$ whenever $uv \in E(G)$. Then for any $N$-vertex, $d$ regular graph $G$,
\[
	\sigma_S(G) \leq \sigma_S(K_{d,d})^{N/(2d)}.
\]
\end{theorem}

\begin{proof}
Observe that $\sigma_S(G) = \hom^\Lambda(G, H_{A,t})$, where $A$ is the multiset $\{\alpha(f(s)): s \in S\}$. Then the inequality is equivalent to $H_{A,t}$ being wGT, which is true by Corollary~\ref{cor:H_n-wGT}.
\end{proof}

Finally we give a weighted generalization of our result on the stable set polytope.

\begin{theorem}
For any $N$-vertex, $d$-regular graph $G$, and any Riemann-integrable function $\tau : [0, 1] \to [0, \infty)$, we have
\begin{equation} \label{eq:estab-integral}
	\int\limits_{\ESTAB(G)} \prod_{v \in V(G)} \tau(x_v) \ d \bx
	\leq
	\( \int\limits_{\ESTAB(K_{d,d})} \prod_{v \in V(K_{d,d})} \tau(x_v) \ d \bx \)^{N/(2d)}.
\end{equation}
\end{theorem}

\begin{proof}
Define a vector of weights $\Lambda_n$ on $H_n$ by $\lambda_i = \tau(i/n)$ for $i \in V(H_n) = \{0, 1, \dots, n\}$. We have
\begin{multline} \label{eq:estab-riem}
	\hom^{\Lambda_n}(G, H_n) = \sum_{f\in \Hom(G, H_n)} w(f) 
	\\= \sum_{f\in \Hom(G, H_n)} \prod_{v \in V(G)} \tau\(\frac{f(v)}{n} \)							 
							 = \sum_{\bx \in \ESTAB(G) \cap \frac{1}{n}\ZZ^{V(G)}} \(\prod_{v \in V(G)} \tau (x_v)\),
\end{multline}
where the last step uses the bijective correspondence $\Hom(G, H_n) \cong \c I(G, n) \cong (n\ESTAB(G)) \cap \ZZ^{V(G)}$. By Riemann sum approximation, 
\begin{equation} \label{eq:estab-riem2}
	\lim_{n\to \infty} n^{-\abs{V(G)}} \hom^{\Lambda_n}(G, H_n) = \int\limits_{\ESTAB(G)} \prod_{v \in V(G)} \tau(x_v) \ d \bx.
\end{equation}
Since $H_n$ is wGT, we may apply \eqref{eq:wGT-def} to $H_n$ and $\Lambda_n$ to obtain
\begin{equation} \label{eq:estab-riem3}
	n^{-N} \hom^{\Lambda_n}(G, H_n) \leq  \(n^{-2d} \hom^\Lambda(K_{d,d}, H)\)^{N/(2d)}.
\end{equation}
Letting $n \to \infty$ in \eqref{eq:estab-riem3} and applying \eqref{eq:estab-riem2} yields the result.
\end{proof}

Using Theorem \ref{thm:stab=estab} we obtain the following result about the stable set polytope.

\begin{corollary}
For any $N$-vertex, $d$-regular graph $G$, and any Riemann-integrable function $\tau : [0, 1] \to [0, \infty)$, we have
\begin{equation}
	\int\limits_{\STAB(G)} \prod_{v \in V(G)} \tau(x_v) \ d \bx
	\leq
	\( \int\limits_{\STAB(K_{d,d})} \prod_{v \in V(K_{d,d})} \tau(x_v) \ d \bx \)^{N/(2d)}.
\end{equation}
\end{corollary}

Setting $\tau = 1$ yields Theorem \ref{thm:iV} as a special case.

\subsection*{Acknowledgments}
The author would like to thank Michel Goemans for his mentorship throughout this project. The author would also like to thank David Galvin for initially suggesting the problem. This research was partially supported by the MIT Undergraduate Research Opportunities Program.

\bibliographystyle{amsplain}
\bibliography{references}

\end{document}